\documentclass[12pt,twoside]{article}

\usepackage{amsmath,amssymb,amscd,amsthm}
\usepackage{mathrsfs}
\usepackage{latexsym}
\usepackage{indentfirst}

\usepackage{graphicx}
\usepackage{color}

\usepackage{cite}
\usepackage{enumerate}
\usepackage{needspace}
\usepackage{authblk}

\usepackage[colorlinks=false,pdfborder={0 0 0}]{hyperref}
\hypersetup{
  colorlinks=true,
  linkcolor=blue,
  citecolor=blue,
  urlcolor=blue
}

\allowdisplaybreaks[4]

\numberwithin{equation}{section}

\newtheorem{theorem}{Theorem}[section]
\newtheorem{corollary}{Corollary}[section]
\newtheorem{lemma}{Lemma}[section]
\newtheorem{proposition}{Proposition}[section]

\newtheorem{remark}{Remark}[section]

\makeatletter \oddsidemargin 0.0in \evensidemargin \oddsidemargin
\begin{document}
\raggedbottom

\title{Exact lattice resonances and a sharp second Picard obstruction for the two-dimensional inviscid primitive equations}

\author{%
  {\large Fengjie Luo}\thanks{Luofengjie664@163.com, (F. Luo)}
 \ and {\large Zhengdong Du}\thanks{
    corresponding author: zdu1985@qq.com, (Z. Du)}.\\
  {\small Department of Mathematics, Sichuan University}\\
  {\small Chengdu, Sichuan 610064, P. R. China}
}

\date{}

\begingroup
\makeatletter
\renewcommand{\@maketitle}{%
  \newpage
  \null
  \vspace*{-1.5cm}%
  {\centering
    \let\footnote\thanks
    {\LARGE \@title \par}%
    \vskip 0.8em%
    {\large
      \lineskip .5em%
      \@author\par}%
  }%
  \par
  \vskip 0.4em%
}
\makeatother
\setcounter{page}{1}
\maketitle
\endgroup
\thispagestyle{empty}

\begingroup
\renewenvironment{abstract}{%
  \small
  \quote
  \setlength{\parskip}{0pt}%
  {\centering\bfseries\abstractname\par}%
  \nobreak\vspace{0.5em}\nobreak
  \noindent\ignorespaces
}{%
  \par\endquote
}

\begin{abstract}
\noindent We study periodic perturbations of a rest state with constant stable stratification for the two-dimensional non-rotating primitive equations
without viscosity or thermal diffusion on a fixed rectangular torus. Fourier resonance analysis and Sobolev estimates show that, despite
unitary linear evolution, the polarized second Picard operator admits no bilinear bound losing less than one derivative at any Sobolev order
and any fixed nonzero time for real data with zero vertical means. The loss is sharp for \(s\ge2\), with matching bounds on every
sufficiently high physical dyadic shell. We classify exact lattice resonances in all signed channels and compute
their polarized coefficients. A mixed-branch family with fixed vertical frequencies isolates diagnostic transport and shows that finite vertical
regularity cannot reduce the loss below one horizontal derivative in tensor-product Sobolev spaces.
In contrast, for \(s\ge2\), the horizontal mean of the second Picard operator admits a time-uniform one-derivative bound.
In interaction variables, the response divided by time converges to the resonant operator, with an \(O(|t|^{-1})\) error after a four-derivative loss for \(s\ge4\).
Finally, among real linear spaces defined solely by a common Fourier-support constraint on both components, the maximal spaces on
which the full nonlinearity vanishes for all smooth states are precisely those supported on primitive lattice rays. On these rays, we construct
explicit global solutions and obtain an exact criterion for uniform strict stratification.

\par\vspace{0.4cm}

\par\vspace{0.4cm}
\noindent\textbf{AMS (2020) subject classification:} 35Q86, 35B30, 35B34, 76B55, 76B70.
\end{abstract}
\endgroup

\clearpage
\flushbottom

\baselineskip 15pt
\parskip 10pt

\section{Introduction}

Existence, stability, and the validity of reduced models are basic
questions in mathematical fluid mechanics. For large-scale oceanic and
atmospheric flows, the small ratio of vertical to horizontal length
scales leads to the primitive equations, in which hydrostatic balance
replaces vertical momentum evolution while the full vertical dependence
is retained. Lions, Temam, and Wang developed a mathematical framework
for these equations \cite{LionsTemamWang1992}, and Cao and Titi proved
global well-posedness of strong solutions to the three-dimensional system
with viscosity and thermal diffusion \cite{CaoTiti2007}.
The hydrostatic approximation was justified for anisotropic viscous flows
by Az\'erad and Guill\'en through weak convergence
\cite{AzeradGuillen2001}. Li and Titi obtained strong convergence globally
and uniformly in time \cite{LiTiti2019}. In a two-dimensional strip,
Paicu, Zhang, and Zhang established global well-posedness and the
hydrostatic limit for small data analytic in the tangential variable
\cite{PaicuZhangZhang2020}.

Without viscosity, the diagnostic recovery of vertical velocity loses a
horizontal derivative, and the available Cauchy theory depends on
additional structure or regularity. For homogeneous hydrostatic Euler
flow in an impermeable channel, Brenier constructed local smooth
solutions with strictly convex horizontal velocity profiles and constant
boundary slopes \cite{Brenier1999}. Masmoudi and Wong proved local
Sobolev existence and uniqueness under the Rayleigh condition
\(\partial_z^2u_0\ge c_0>0\), removing the boundary-slope restriction
\cite{MasmoudiWong2012}. Analytic local well-posedness on bounded domains
was established by Kukavica et al.\ \cite{KukavicaTemamVicolZiane2011},
while vertical viscosity and convexity permit a Gevrey theory
\cite{GerardVaretMasmoudiVicol2020}. Velocity convexity is distinct
from strict temperature stratification and is incompatible with vertical
periodicity. Conversely, Renardy proved
linear ill-posedness near certain shear flows \cite{Renardy2009}, and
Han-Kwan and Nguyen established nonlinear Sobolev ill-posedness using an
unbounded unstable spectrum \cite{HanKwanNguyen2016}. For hydrostatic
Euler--Boussinesq flow, Bianchini, Coti Zelati, and Ertzbischoff obtained
nonlinear ill-posedness and failure of the hydrostatic limit near
suitable stably stratified shears violating the Miles--Howard condition
\cite{BianchiniCotiZelatiErtzbischoff2025}. Stable singularity formation
has also been studied under symmetry and profile assumptions, first in
the isothermal case \cite{CollotIbrahimLin2024} and then with nonconstant
temperature \cite{IbrahimLinQianTiti2026}. In a different formulation,
Bianchini and Duch\^ene used stable stratification and isopycnal
coordinates to establish Sobolev well-posedness and the hydrostatic limit
with thickness diffusivity but no momentum viscosity
\cite{BianchiniDuchene2024}.

Stable stratification can also provide long-time control through internal
gravity waves or, in the presence of shear, phase mixing. The classical
Miles--Howard criterion gives spectral stability for stratified shear
flows when the local Richardson number is at least \(\frac{1}{4}\)
\cite{Miles1961,Howard1961}. Near a motionless stratified state in
\(\mathbb R^2\), Elgindi and Widmayer used anisotropic dispersion to
extend the stability time for the inviscid Boussinesq system
\cite{ElgindiWidmayer2015}. Jurja and Widmayer obtained the
\(O(\varepsilon^{-2})\) timescale for small, localized Sobolev
perturbations satisfying additional weighted and anisotropic bounds,
using refined dispersive estimates and a decomposition of nonlinear
interactions \cite{JurjaWidmayer2026}. In three dimensions, Takada
proved existence on any prescribed finite interval for sufficiently
strong stratification, together with convergence to layered Euler
dynamics \cite{Takada2019}. Rotation yields related results: Ghoul et
al.\ established analytic local well-posedness for rotating isothermal
primitive equations and extended the lifespan for well-prepared data
under fast rotation \cite{GhoulIbrahimLinTiti2022}. Guo, Pausader, and
Widmayer constructed global small localized axisymmetric Euler flows
near uniform rotation
\cite{GuoPausaderWidmayer2023}. For the Boussinesq system with velocity
viscosity but no thermal diffusion, Tao et al.\ established stability for
periodic perturbations of a linearly stratified hydrostatic equilibrium
\cite{TaoWuZhaoZheng2020}. Stability near Couette flow for Gevrey data
was proved in \cite{MasmoudiSaidHouariZhao2022}. For small Gevrey perturbations of inviscid
stratified Couette flow, Bedrossian et al.\ proved inviscid damping and
quantified gradient growth up to times of order \(\varepsilon^{-2}\)
under the strict Miles--Howard condition
\cite{BedrossianBianchiniCotiZelatiDolce2023}. Dispersion and mixing also
improve the stability threshold in the viscous three-dimensional
Boussinesq problem \cite{CotiZelatiDelZottoWidmayer2026}.

The analysis of oscillatory evolution requires both the resonance
geometry and the coefficients of the nonlinear interactions. Fast
singular limits and wave averaging provide a general framework for this
analysis \cite{Schochet1994,EmbidMajda1996}. For deterministic equations,
low-order Duhamel terms test the estimates required by an iteration
scheme. This approach underlies the Benjamin--Ono obstructions of
Molinet, Saut, and Tzvetkov \cite{MolinetSautTzvetkov2001} and the
first-iterate analysis of Bejenaru and Tao \cite{BejenaruTao2006}.
Related questions of data-to-solution regularity were studied for
dispersive equations by Kenig, Ponce, and Vega \cite{KenigPonceVega2001}
and for Euler flow by Himonas and Misio\l ek \cite{HimonasMisiolek2010}.
In wave kinetic theory, Shavit, B\"uhler, and Shatah derived kinetic
equations and turbulent spectra for weakly interacting internal waves
\cite{ShavitBuhlerShatah2024,ShavitBuhlerShatah2025}.
Xiang established weighted \(L^\infty\) bounds for the Boussinesq
collision operator with angular cutoffs and local well-posedness of the
associated wave kinetic equation \cite{Xiang2026}.
For the hydrostatic model, the same work derives finite-box interaction
coefficients and identifies collinear cancellation. A near-horizontal
high--low construction obstructs removal of the horizontal angular
cutoff in the radially weighted \(L^\infty\) spaces considered there
\cite[Section~8.2, Remark~8.4, and Appendix~A.5]{Xiang2026}.
These kinetic results concern continuous spectral densities. They do
not directly provide Sobolev estimates for the discrete amplitudes of
the deterministic equation on a fixed torus.

This leaves a specific question for a motionless, uniformly stratified
hydrostatic state on a doubly periodic domain: does time integration of
the linear oscillations compensate for the diagnostic derivative loss?
Linear unitarity alone does not answer this question. The analysis must
also distinguish active resonances from support constraints that cancel
the full nonlinearity. We address these questions for the equation
determined by periodic compatibility, with fixed stratification and
torus dimensions, without viscosity, thermal diffusion, or angular
cutoffs.

Let \(L_x,L_z>0\), and set
\[
 \Omega:=\mathbb T_x\times\mathbb T_z,
 \qquad \mathbb T_x:=\mathbb R/(L_x\mathbb Z),
 \qquad \mathbb T_z:=\mathbb R/(L_z\mathbb Z).
\]

For a time interval \(I\subset\mathbb R\) containing the origin, we
consider the two-dimensional, non-rotating system on \(I\times\Omega\):
\begin{equation}
\label{eq:1.2.1}
\left\{
\begin{aligned}
 &\partial_tv+v\partial_xv+w\partial_zv+\partial_xp=0,\\
 &\partial_zp=g\beta\theta,\\
 &\partial_t\theta+v\partial_x\theta+w\partial_z\theta+\gamma w=0,\\
 &\partial_xv+\partial_zw=0.
\end{aligned}
\right.
\end{equation}
Here \(v,w,\theta,p\) denote the horizontal velocity, vertical velocity,
temperature perturbation, and pressure perturbation, respectively, and
\(g,\beta,\gamma>0\). All perturbations are real-valued and periodic in
both spatial variables. The background temperature
\(T_s(z)=T_0+\gamma z\) and pressure
\(\pi_s(z)=-gz+\frac{g\beta\gamma}{2}z^2\) are not vertically periodic.
We impose periodicity only on their perturbations and set
\(N:=\sqrt{g\beta\gamma}\), the Brunt--V\"ais\"al\"a frequency.
Strict stratification means
\(\operatorname*{ess\,inf}_{\Omega}(\gamma+\partial_z\theta)>0\).

Our first result determines the sharp derivative loss of the polarized
second Picard operator \(\mathcal B_t\). Let \(X^s\) be the real
\(H^s\) phase space with zero vertical means, as defined in
Section~\ref{sec:2-preliminaries-main-results}. For every \(s\in\mathbb R\) and fixed \(t\ne0\),
\(\mathcal B_t\) has no bounded bilinear extension from
\(X^s\times X^s\) to \(X^{s-\delta}\) when \(0\le\delta<1\), whereas
it is bounded into \(X^{s-1}\) for \(s\ge2\)
(Theorem~\ref{thm:2.2.1-second-picard-obstruction}). Thus the quadratic obstruction occurs at a
linearly unitary rest state, without exponentially growing linear modes.
It concerns iteration estimates and does not by itself establish
Hadamard ill-posedness of the nonlinear problem.

The resonance classification in Theorem~\ref{thm:2.2.2-exact-lattice-resonances} covers all four
ordered branch combinations on the signed integer lattice, including
zero horizontal input and output frequencies, and determines their
polarized coefficients. A mixed-branch construction isolates diagnostic
transport at fixed vertical frequencies
(Corollary~\ref{cor:4.2.1-diagnostic-mixed-resonances}). It shows that no finite choice
of vertical Sobolev indices compensates for a horizontal derivative loss
below one in the full tensor-product spaces
(Corollary~\ref{cor:4.2.2-anisotropic-picard-obstruction}). Matching bounds on every
sufficiently high physical dyadic shell give a scale-localized version of
the obstruction (Proposition~\ref{prop:4.2.1-dyadic-picard-bounds}).
These conclusions concern the complete periodic
Duhamel operator on real-valued data, rather than a cutoff collision
operator for spectral densities.

Two positive estimates complement this obstruction
(Theorem~\ref{thm:2.2.4-mean-resonant-averaging}). The horizontal mean of
\(\mathcal B_t\) is bounded from \(X^s\times X^s\) to \(X^{s-1}\),
uniformly in time, for \(s\ge2\). For each fixed pair of inputs,
\(t^{-1}e^{-t\mathcal L}\mathcal B_t\) converges in \(X^{s-1}\) to the
operator obtained by retaining only exact resonances. For \(s\ge4\),
the error is \(O(|t|^{-1})\) in \(X^{s-4}\), uniformly on bounded
sets of inputs. Thus the resonant operator gives the leading
long-time quadratic response, while its horizontal mean vanishes.

Theorem~\ref{thm:2.2.3-maximal-transparent-rays} characterizes the common Fourier supports on which
both the quadratic and cubic terms vanish. The maximal nonzero real
linear spaces defined by such support constraints are exactly the
primitive lattice rays, with no amplitude relation between the two
components. For \(s>2\), this cancellation gives explicit global
ray-supported solutions (Corollary~\ref{cor:5.3.1-fourier-ray-dynamics})
and an exact criterion for uniform strict stratification
(Proposition~\ref{prop:5.3.1-uniform-stratification}). The Picard bounds
also yield obstructions to quadratic iteration
(Theorem~\ref{thm:5.1.1-quadratic-iteration-obstruction}) and, under
separate existence and Duhamel-consistency assumptions, twice
differentiability of trajectory maps
(Corollary~\ref{cor:5.1.1-conditional-non-c2}). The fixed-time
obstruction persists under bounded quadratic changes of variables
(Corollary~\ref{cor:5.2.1-quadratic-change-obstruction}).

The main analytic difficulty is the interaction between diagnostic
reconstruction and wave oscillations. For the mean-zero part
\(w^0=-\partial_z^{-1}\partial_xv\), Fourier indices \((n,m)\) with
\(m\ne0\) give
\[
 \widehat{w^0}(n,m)=-\frac{L_z n}{L_xm}\widehat v(n,m)
 \quad\text{and}\quad
 \omega_\sigma(n,m)=\sigma N\frac{L_z n}{L_xm},
 \qquad \sigma\in\{\pm1\}.
\]
The same unbounded ratio occurs in the diagnostic velocity and the wave
frequency. In contrast, the nonhydrostatic frequency has magnitude
\(N\frac{|\xi_x|}{|\xi|}\). Individual modes on the fixed torus prevent uniform
dispersive decay, and the rest state provides no shear mixing.
Moreover, periodic compatibility fixes
\(\bar w=-\frac{1}{\gamma}\partial_x\overline{v\theta}\), where the
bar denotes vertical averaging. We retain this correction and the
corresponding pressure correction when deriving the exact quadratic and
cubic terms in Section~\ref{sec:2-preliminaries-main-results}.
At the endpoint \(s=2\), mixed-norm estimates adapted to vertical
integration control the products without using an algebra estimate in
\(H^1(\Omega)\).

A second issue is to distinguish a resonant phase from an active
interaction. With the phase \(\Phi\) and polarized coefficient
\(\widetilde q\) defined in Section~\ref{sec:2-preliminaries-main-results}, the frequencies
\(k_1=(1,1)\), \(k_2=(-M^2,M)\), and \(k=k_1+k_2\), for \(M\ge2\),
satisfy
\[
 \Phi_{\sigma\sigma\sigma}(k;k_1,k_2)=0
 \quad\text{and}\quad
 \widetilde q_{\sigma\sigma\sigma}(k;k_1,k_2)
 =\frac{i\kappa_x}{2\sqrt\gamma}(M^2-1),
\]
where \(\kappa_x:=\frac{2\pi}{L_x}\). The resonant time integral equals \(t\), while the coefficient and the
high input and output frequencies have size comparable to \(M^2\).
The lower bound still requires real-valued inputs: their conjugate modes
must not cancel the selected output. We separate these outputs and use
two-sided comparisons of the Sobolev weights, which remain valid for
negative indices. The mixed-branch construction requires a separate
calculation, since the diagnostic contributions cancel in the
same-branch family after polarization.
For nonzero phases, the integer numerator in the phase formula gives
an explicit lower bound and a quantitative averaging estimate with
finite derivative loss. At zero horizontal output, direct division
of the polarized coefficient by the phase gives the stronger
one-derivative bound, uniformly in time.

The support classification must account for the averaged terms as well
as local transport. For smooth vertically mean-zero functions supported on a
single lattice ray generated by \((n_0,m_0)\), with \(m_0\ne0\),
\[
 \partial_z^{-1}\partial_xf=\lambda_0f
 \quad\text{and}\quad
 \partial_x\overline{fg}=0,
 \qquad \lambda_0:=\frac{L_zn_0}{L_xm_0}.
\]
These identities remove both local transport and the quadratic and cubic
averaged corrections. Conversely, two noncollinear allowed frequencies
produce a real state with a nonzero quadratic output. The argument thus
classifies the common-support spaces for the full nonlinearity, rather
than only the zero set of a single transport multiplier.

The paper is organized as follows.
Section~\ref{sec:2-preliminaries-main-results} derives the diagnostic formulation, develops the
linear theory, and states the main results.
Section~\ref{sec:3-quadratic-wave-interactions} contains the wave multipliers, the
Fourier--Duhamel representation, the endpoint multilinear estimates,
and the phase-division bounds.
Section~\ref{sec:4-exact-resonances-picard-obstruction} proves the main
theorems, the anisotropic and dyadic refinements, and the horizontal-mean
and resonant-averaging estimates.
Section~\ref{sec:5-iteration-obstructions-fourier-rays} treats quadratic iteration, conditional
trajectory-map regularity, quadratic changes of variables, and exact
dynamics on Fourier rays.


\section{Preliminaries and main results}
\label{sec:2-preliminaries-main-results}

We derive the diagnostic formulation, introduce the functional setting
and wave variables, and establish the properties of the linear group.
The main results are stated in Subsection~\ref{subsec:2.2-main-results}.

\subsection{Preliminaries}
\label{subsec:2.1-preliminaries}

For a periodic function \(f\), write
\[
 \bar f(x):=\frac{1}{L_z}\int_0^{L_z}f(x,z)\,\mathrm dz
 \quad\text{and}\quad
 \langle f\rangle_x(z):=\frac{1}{L_x}\int_0^{L_x}f(x,z)\,\mathrm dx.
\]
Each average is identified with its constant extension in the averaged
variable. Set \(\mathcal Pf:=f-\bar f\). For \(\bar f=0\), let
\(\partial_z^{-1}f\) denote the unique periodic primitive with zero
vertical mean.

Let \((v,w,\theta,p)\) be a smooth real-valued periodic solution of
\eqref{eq:1.2.1}. Vertical averaging of hydrostatic balance and
incompressibility implies \(\bar\theta=0\) and \(\partial_x\bar v=0\).
The conservative momentum equation preserves the spatial mean of
\(v\). Hence \(\bar v\) is constant in both space and time, and a
horizontal Galilean transformation reduces the problem to the invariant
class
\begin{equation}
\label{eq:2.1.1}
        \bar v=0,
        \qquad
        \bar\theta=0.
\end{equation}

Incompressibility and the vertically averaged temperature equation imply
\[
 w=-\partial_z^{-1}\partial_xv+\bar w
 \quad\text{and}\quad
 \gamma\bar w=-\partial_x\overline{v\theta}.
\]
Thus
\begin{equation}
\label{eq:2.1.2}
        w=-\partial_z^{-1}\partial_xv
        -\frac{1}{\gamma}\partial_x\overline{v\theta}.
\end{equation}
Thus periodic compatibility fixes the vertical mean of \(w\), which has
no independent initial datum.

Similarly, hydrostatic balance and the vertically averaged momentum
equation imply
\[
 p=g\beta\partial_z^{-1}\theta+\bar p
 \quad\text{and}\quad
 \partial_x\bar p=-\partial_x\overline{v^2}.
\]
The normalization \(\langle\bar p\rangle_x=0\) then determines
\begin{equation}
\label{eq:2.1.3}
        p=g\beta\partial_z^{-1}\theta
        -\overline{v^2}
        +\bigl\langle\overline{v^2}\bigr\rangle_x.
\end{equation}
Equations \eqref{eq:2.1.2}--\eqref{eq:2.1.3} uniquely determine
\(w,p\) under this normalization.

For smooth scalar functions \(f,g\) with zero vertical mean, define
\begin{equation}
        b(f,g):=(\partial_z^{-1}\partial_xf)\partial_zg-f\partial_xg.
\end{equation}
Integration by parts gives
\(\overline{b(f,g)}=-\partial_x\overline{fg}\), so
\begin{equation}
\label{eq:2.1.5}
        \mathcal P b(f,g)
        =(\partial_z^{-1}\partial_xf)\partial_zg-f\partial_xg
        +\partial_x\overline{fg}.
\end{equation}
For smooth pairs \(U=(v,\theta)\) and \(U_j=(v_j,\theta_j)\) with
zero vertical mean, define the linear operator
\begin{equation}
\label{eq:2.1.6}
        \mathcal L(v,\theta)
        :=(-g\beta\partial_z^{-1}\partial_x\theta,\gamma\partial_z^{-1}\partial_xv),
\end{equation}
and the multilinear terms
\begin{equation}
\label{eq:2.1.7}
        \mathcal N_2(U_1,U_2)
        :=\bigl(\mathcal P b(v_1,v_2),
                 \mathcal P b(v_1,\theta_2)\bigr),
\end{equation}
and
\begin{equation}
\label{eq:2.1.8}
 \mathcal N_3(U_1,U_2,U_3)
 :=\frac{1}{\gamma}\partial_x\overline{v_1\theta_2}
   \bigl(\partial_zv_3,\partial_z\theta_3\bigr).
\end{equation}
The first argument of \(\mathcal N_2\) supplies the advecting velocity.

Substituting \eqref{eq:2.1.2} and \eqref{eq:2.1.3} into
\eqref{eq:1.2.1} yields
\begin{equation}
\label{eq:2.1.9}
\left\{
\begin{aligned}
 &\partial_tv=-g\beta\partial_z^{-1}\partial_x\theta
 +\mathcal P b(v,v)
 +\frac{1}{\gamma}\partial_x\overline{v\theta}\,\partial_zv,\\
 &\partial_t\theta=\gamma\partial_z^{-1}\partial_xv
 +\mathcal P b(v,\theta)
 +\frac{1}{\gamma}\partial_x\overline{v\theta}\,\partial_z\theta.
\end{aligned}
\right.
\end{equation}
Conversely, \eqref{eq:2.1.9} preserves \eqref{eq:2.1.1}. The linear and
projected terms have zero vertical mean. The cubic terms also have zero
vertical mean, since their common prefactor is independent of \(z\).
Reconstruction by \eqref{eq:2.1.2}--\eqref{eq:2.1.3} then recovers
\eqref{eq:1.2.1} with the prescribed pressure normalization. The two
formulations are therefore equivalent for smooth solutions in this class.

Testing the momentum and temperature equations against \(\gamma v\)
and \(g\beta\theta\), respectively, eliminates the transport terms by
periodicity and incompressibility. Hydrostatic balance cancels the
remaining coupling terms:
\[
 -\gamma\int_\Omega v\partial_xp
 -g\beta\gamma\int_\Omega\theta w
 =\gamma\int_\Omega w\partial_zp
 -g\beta\gamma\int_\Omega\theta w=0.
\]
Consequently,
\begin{equation}
\label{eq:2.1.10}
 \frac{\mathrm d}{\mathrm dt}
 \left(
 \frac{\gamma}{2}\|v(t)\|_{L^2}^2
 +\frac{g\beta}{2}\|\theta(t)\|_{L^2}^2
 \right)=0.
\end{equation}
In terms of \(U\), the closed system takes the exact polynomial form
\begin{equation}
\label{eq:2.1.11}
        \partial_tU
        =\mathcal L U+\mathcal N_2(U,U)+\mathcal N_3(U,U,U).
\end{equation}
The cubic term has zero second differential at the origin. Thus, for
smooth directions \(H_1,H_2\), the second differential of the full
nonlinearity is the polarization of \(\mathcal N_2(U,U)\):
\begin{equation}
\label{eq:2.1.12}
 D^2\bigl(\mathcal N_2(U,U)+\mathcal N_3(U,U,U)\bigr)\big|_{U=0}[H_1,H_2]
 =\mathcal N_2(H_1,H_2)+\mathcal N_2(H_2,H_1).
\end{equation}

We now introduce the Fourier and Sobolev conventions. Unless stated
otherwise, norms are taken over \(\Omega\), and the parameters
\(L_x,L_z,g,\beta,\gamma\) are fixed. For nonnegative quantities
\(f,g\), the notation \(f\lesssim g\) means \(f\le Cg\), where \(C>0\)
depends only on the fixed parameters and the indicated Sobolev indices,
unless stated otherwise. We write \(f\gtrsim g\) when \(g\lesssim f\),
and \(f\simeq g\) when both inequalities hold.

Set
\[
 \kappa_x:=\frac{2\pi}{L_x},\qquad
 \kappa_z:=\frac{2\pi}{L_z},\qquad
 \mathbb Z_*:=\mathbb Z\setminus\{0\},\qquad
 K:=\mathbb Z\times\mathbb Z_*.
\]
For \(k=(n,m)\in\mathbb Z^2\), write \(\xi_k:=(\kappa_xn,\kappa_zm)\).

For \(f\in L^2(\Omega)\), define
\begin{equation}
 \widehat f(n,m)
 :=\frac{1}{L_xL_z}\int_0^{L_x}\int_0^{L_z}
 f(x,z)e^{-i(\kappa_xnx+\kappa_zmz)}\,\mathrm dz\,\mathrm dx
\end{equation}
and
\begin{equation}
 f(x,z)=\sum_{(n,m)\in\mathbb Z^2}
 \widehat f(n,m)e^{i(\kappa_xnx+\kappa_zmz)},
\end{equation}
where the series converges in \(L^2(\Omega)\). For \(s\in\mathbb R\),
we use the inhomogeneous Sobolev norm
\begin{equation}
\label{eq:2.1.15}
        \|f\|_{H^s}^2
        :=L_xL_z\sum_{k\in\mathbb Z^2}
        \langle\xi_k\rangle^{2s}|\widehat f(k)|^2,
        \qquad
        \langle\xi\rangle:=(1+|\xi|^2)^{\frac{1}{2}}.
\end{equation}
The averaging operators and \(\mathcal P\) are bounded Fourier
projections on every \(H^s\). For negative indices, the mean constraints
are understood distributionally.
On the subspace \(\bar f=0\), the inverse vertical derivative is given by
\[
        \widehat{\partial_z^{-1}f}(n,m)
        =\frac{1}{i\kappa_zm}\widehat f(n,m),
        \qquad m\in\mathbb Z_*.
\]
The hydrostatic multiplier \(\partial_z^{-1}\partial_x\) has the real
symbol
\begin{equation}
        \widehat{\partial_z^{-1}\partial_xf}(k)=\lambda(k)\widehat f(k),
        \qquad
        \lambda(k):=\frac{\kappa_xn}{\kappa_zm}
        =\frac{L_z}{L_x}\frac{n}{m},
        \qquad k=(n,m)\in K.
\end{equation}
This multiplier commutes with spatial derivatives and satisfies
\(|\lambda(k)|\le\kappa_z^{-1}\langle\xi_k\rangle\). Hence, for
\(f\in H^s\) with \(\bar f=0\),
\begin{equation}
\label{eq:2.1.17}
        \|\partial_z^{-1}\partial_xf\|_{H^{s-1}}
        \le C\|f\|_{H^s},
        \qquad s\in\mathbb R.
\end{equation}
It is not bounded on \(H^s\), since \(|\lambda(n,1)|\to\infty\) as
\(|n|\to\infty\).

The real energy space is
\[
 X^s:=\bigl\{(v,\theta)\in H^s(\Omega;\mathbb R)^2:
                   \bar v=\bar\theta=0\bigr\},
\]
with norm
\[
 \|U\|_{X^s}^2:=\gamma\|v\|_{H^s}^2+g\beta\|\theta\|_{H^s}^2.
\]
The mean constraints are equivalent to
\(\widehat v(n,0)=\widehat\theta(n,0)=0\) for every \(n\in\mathbb Z\).
Write \(X^\infty:=\bigcap_{r\ge0}X^r\). The complexification
\(X^s_{\mathbb C}\) of \(X^s\) consists of complex-valued pairs with the
same mean constraints and inner product
\begin{equation}
 \langle U_1,U_2\rangle_{X^s_{\mathbb C}}
 :=\gamma\langle v_1,v_2\rangle_{H^s}
   +g\beta\langle\theta_1,\theta_2\rangle_{H^s}.
\end{equation}
For \(s>2\) and \(0<\kappa<\gamma\), set
\begin{equation}
        X_\kappa^s
        :=\left\{U=(v,\theta)\in X^s:
        \operatorname*{ess\,inf}_{\Omega}
        (\gamma+\partial_z\theta)\ge\kappa\right\}.
\end{equation}
Sobolev embedding and \(\kappa<\gamma\) imply that \(X_\kappa^s\)
contains an open ball about the origin in \(X^s\). Write
\(X^\infty_{\mathbb C}:=\bigcap_{r\ge0}X^r_{\mathbb C}\), and extend
all multilinear maps by complex linearity in each argument. An overbar
denotes vertical averaging on spatial functions and complex conjugation
on Fourier coefficients.

The anisotropic estimates use tensor-product Sobolev spaces. For
\(s_x,s_z\in\mathbb R\), let
\(H^{s_x,s_z}(\Omega;\mathbb C)\) be the completion of the complex
trigonometric polynomials with respect to the norm
\begin{equation}
\label{eq:2.1.20}
 \|f\|_{H^{s_x,s_z}}^2
 :=L_xL_z\sum_{(n,m)\in\mathbb Z^2}
 \langle\kappa_xn\rangle^{2s_x}
 \langle\kappa_zm\rangle^{2s_z}|\widehat f(n,m)|^2,
\end{equation}
where \(\langle\cdot\rangle\) has the same meaning for scalar frequencies.
Write \(H^{s_x,s_z}(\Omega;\mathbb R)\) for the real subspace, and set
\begin{equation}
 X^{s_x,s_z}
 :=\bigl\{(v,\theta)\in H^{s_x,s_z}(\Omega;\mathbb R)^2:
                   \bar v=\bar\theta=0\bigr\}
\end{equation}
and
\begin{equation}
 \|U\|_{X^{s_x,s_z}}^2
 :=\gamma\|v\|_{H^{s_x,s_z}}^2
   +g\beta\|\theta\|_{H^{s_x,s_z}}^2.
\end{equation}
The mean constraints are imposed in the Fourier sense. Symmetric Fourier
truncation shows that constrained trigonometric polynomials are dense
in each of the constrained phase spaces defined above.

We equip \(\mathcal L\) from \eqref{eq:2.1.6} with its maximal Fourier
multiplier domain
\begin{equation}
\label{eq:2.1.23}
 \operatorname{Dom}_{X^s_{\mathbb C}}(\mathcal L)
 :=\bigl\{(v,\theta)\in X^s_{\mathbb C}:
          \partial_z^{-1}\partial_xv,\partial_z^{-1}\partial_x\theta\in H^s(\Omega;\mathbb C)\bigr\}.
\end{equation}
The real operator is the restriction of \(\mathcal L\) to
\(\operatorname{Dom}_{X^s_{\mathbb C}}(\mathcal L)\cap X^s\).

In the energy variables, the linear equation
\(\partial_tU=\mathcal LU\) becomes
\[
 \partial_t
 \begin{pmatrix}
  \sqrt\gamma\,\widehat v(k)\\[1mm]
  \sqrt{g\beta}\,\widehat\theta(k)
 \end{pmatrix}
 =N\lambda(k)
 J
 \begin{pmatrix}
  \sqrt\gamma\,\widehat v(k)\\[1mm]
  \sqrt{g\beta}\,\widehat\theta(k)
 \end{pmatrix},
 \qquad
 J:=\begin{pmatrix}0&-1\\[1mm]1&0\end{pmatrix}.
\]
Since \(J^2=-I_2\), the Fourier solution is
\begin{equation}
\label{eq:2.1.24}
 \begin{pmatrix}
  \sqrt\gamma\,\widehat v(t,k)\\[1mm]
  \sqrt{g\beta}\,\widehat\theta(t,k)
 \end{pmatrix}
 =
 \begin{pmatrix}
  \cos(N\lambda(k)t)&-\sin(N\lambda(k)t)\\[1mm]
  \sin(N\lambda(k)t)& \cos(N\lambda(k)t)
 \end{pmatrix}
 \begin{pmatrix}
  \sqrt\gamma\,\widehat v_0(k)\\[1mm]
  \sqrt{g\beta}\,\widehat\theta_0(k)
 \end{pmatrix}.
\end{equation}
For a pair \(G=(G_v,G_\theta)\), define
\begin{equation}
\label{eq:2.1.25}
 Z_\sigma[G]:=\sqrt\gamma\,G_v+i\sigma\sqrt{g\beta}\,G_\theta,
 \qquad \sigma\in\{\pm1\}.
\end{equation}
We write \(Z_\sigma:=Z_\sigma[U]\) and
\(Z_{j,\sigma}:=Z_\sigma[U_j]\). These variables diagonalize the
linear evolution, with inverse transformation
\begin{equation}
\label{eq:2.1.26}
        v=\frac{Z_++Z_-}{2\sqrt\gamma},
        \qquad
        \theta=\frac{Z_+-Z_-}{2i\sqrt{g\beta}}.
\end{equation}
The signed branch frequencies are
\begin{equation}
        \omega_\sigma(n,m)
        =\sigma N\frac{L_zn}{L_xm},
        \qquad \sigma\in\{\pm1\}.
\end{equation}
The branch label \(\sigma\) does not determine the sign of
\(\omega_\sigma(k)\), since \(\lambda(k)\) may have either sign or
vanish. In these variables, the linear equations and their solutions are
\begin{equation}
\label{eq:2.1.28}
        \partial_tZ_\sigma=i\sigma N\partial_z^{-1}\partial_xZ_\sigma,
        \qquad
        \widehat Z_\sigma(t,k)
        =e^{it\omega_\sigma(k)}\widehat Z_\sigma(0,k).
\end{equation}
The wave transformation also satisfies
\begin{equation}
\label{eq:2.1.29}
        \|U\|_{X^s}^2
        =\frac12\|Z_+\|_{H^s}^2
         +\frac12\|Z_-\|_{H^s}^2.
\end{equation}
The same identity holds on \(X^s_{\mathbb C}\). Hence
\(U\mapsto 2^{-\frac12}(Z_+,Z_-)\) is unitary from \(X^s_{\mathbb C}\)
onto the product of two zero-vertical-mean \(H^s(\Omega;\mathbb C)\)
spaces. The interaction formulas use the unnormalized variables
\(Z_\sigma\).
For real-valued states,
\begin{equation}
\label{eq:2.1.30}
        \widehat Z_\sigma(-k)
        =\overline{\widehat Z_{-\sigma}(k)}.
\end{equation}

The Fourier representation establishes the following properties of the
linear evolution.

\begin{proposition}
\label{prop:2.1.1-unitary-linear-evolution}
Let \(s\in\mathbb R\). The operator \(\mathcal L\), with domain
\eqref{eq:2.1.23}, is densely defined and skew-adjoint on
\(X^s_{\mathbb C}\). It generates the strongly continuous unitary group
\(\{e^{t\mathcal L}\}_{t\in\mathbb R}\) given by \eqref{eq:2.1.24}.
This group leaves \(X^s\) invariant and satisfies, for every
\(U_0\in X^s_{\mathbb C}\),
\begin{equation}
 \|e^{t\mathcal L}U_0\|_{X^s_{\mathbb C}}
 =\|U_0\|_{X^s_{\mathbb C}},
 \qquad t\in\mathbb R.
\end{equation}
\end{proposition}

\begin{proof}
By \eqref{eq:2.1.29}, the map
\(U\mapsto(c_\sigma(k))_{k,\sigma}\), where
\(c_\sigma(k):=2^{-\frac12}\widehat Z_\sigma(k)\), is unitary from
\(X^s_{\mathbb C}\) onto \(\ell_s^2(K;\mathbb C^2)\), equipped with
the norm
\[
 \|c\|_{\ell_s^2}^2
 :=L_xL_z\sum_{k\in K}\sum_{\sigma=\pm1}
       \langle\xi_k\rangle^{2s}|c_\sigma(k)|^2.
\]
In these coordinates, \(\mathcal L\) becomes
\((\mathcal Mc)_\sigma(k)=i\omega_\sigma(k)c_\sigma(k)\), with
maximal domain
\[
 \operatorname{Dom}(\mathcal M)
 =\{c\in\ell_s^2:(\omega_\sigma(k)c_\sigma(k))_{k,\sigma}
                  \in\ell_s^2\}.
\]
Finite sequences belong to this domain and are dense in \(\ell_s^2\).
For \(h\in\operatorname{Dom}(\mathcal M^*)\), testing against a single
frequency and branch identifies
\((\mathcal M^*h)_\sigma(k)=-i\omega_\sigma(k)h_\sigma(k)\).
Conversely, if \((\omega_\sigma h_\sigma)\in\ell_s^2\), the adjoint
identity follows by Cauchy--Schwarz. Therefore,
\[
 \operatorname{Dom}(\mathcal M^*)=\operatorname{Dom}(\mathcal M)
 \quad\text{and}\quad \mathcal M^*=-\mathcal M.
\]
Thus \(\mathcal L\) is densely defined and skew-adjoint.

The multipliers \(e^{it\omega_\sigma(k)}\) define an isometric group.
Dominated convergence establishes strong continuity:
\[
 \|e^{t\mathcal L}U-U\|_{X^s_{\mathbb C}}^2
 =L_xL_z\sum_{k\in K}\sum_{\sigma=\pm1}
 \langle\xi_k\rangle^{2s}|e^{it\omega_\sigma(k)}-1|^2
 |c_\sigma(k)|^2\longrightarrow0
\]
as \(t\to0\). For \(c\in\operatorname{Dom}(\mathcal M)\), the bound
\[
 \left|\frac{e^{it\omega_\sigma(k)}-1}{t}
              -i\omega_\sigma(k)\right|\le2|\omega_\sigma(k)|
\]
identifies the strong derivative as \(\mathcal Mc\) by dominated
convergence. Conversely, convergence of the difference quotients in
\(\ell_s^2\) identifies their limit coordinatewise as
\((i\omega_\sigma c_\sigma)_{k,\sigma}\). Membership of this sequence
in \(\ell_s^2\) is precisely the maximal domain condition. The generator
is therefore \(\mathcal L\). Finally,
\(\omega_\sigma(-k)=\omega_\sigma(k)=-\omega_{-\sigma}(k)\)
shows that the group preserves \eqref{eq:2.1.30} and thus \(X^s\).
\end{proof}

The bound \eqref{eq:2.1.17} also implies
\(X^s_{\mathbb C}\subset
\operatorname{Dom}_{X^{s-1}_{\mathbb C}}(\mathcal L)\).
With the tensor-product weights from \eqref{eq:2.1.20}, the wave
transformation satisfies
\begin{equation}
\label{eq:2.1.32}
 \|U\|_{X^{s_x,s_z}}^2
 =\frac12\|Z_+\|_{H^{s_x,s_z}}^2
   +\frac12\|Z_-\|_{H^{s_x,s_z}}^2
\end{equation}
and
\begin{equation}
 \|e^{t\mathcal L}U\|_{X^{s_x,s_z}}
 =\|U\|_{X^{s_x,s_z}},\qquad t\in\mathbb R.
\end{equation}
Dominated convergence also proves strong continuity on
\(X^{s_x,s_z}\) for all \(s_x,s_z\in\mathbb R\).

It follows from the diagonal representation that, on \(X^s_{\mathbb C}\),
\begin{equation}
 \sigma_{\mathrm p}(\mathcal L)=iN\frac{L_z}{L_x}\mathbb Q,
 \qquad
 \sigma(\mathcal L)=i\mathbb R,
\end{equation}
and
\begin{equation}
 \sigma_{\mathrm c}(\mathcal L)
 =i\mathbb R\setminus iN\frac{L_z}{L_x}\mathbb Q,
 \qquad
 \sigma_{\mathrm r}(\mathcal L)=\varnothing.
\end{equation}
Indeed, the diagonal entries \(i\omega_\sigma(k)\) range over
\(iN\frac{L_z}{L_x}\mathbb Q\), which is exactly the point spectrum.
This set is dense in \(i\mathbb R\). Since the spectrum of a skew-adjoint
operator is closed and contained in \(i\mathbb R\), the full spectrum
is \(i\mathbb R\). For
\(\eta\in\mathbb R\),
\[
 \bigl(\operatorname{Ran}(\mathcal L-i\eta)\bigr)^\perp
 =\ker\bigl((\mathcal L-i\eta)^*\bigr)
 =\ker(\mathcal L-i\eta).
\]
Thus injectivity of \(\mathcal L-i\eta\) implies density of its range,
so the residual spectrum is empty. If
\(\eta\notin N\frac{L_z}{L_x}\mathbb Q\), choose normalized
single-mode eigenvectors \(U_j\) in the \(+\) branch with
\(N\lambda(k_j)\to\eta\). Then
\[
 \|U_j\|_{X^s_{\mathbb C}}=1
 \quad\text{and}\quad
 \|(\mathcal L-i\eta)U_j\|_{X^s_{\mathbb C}}
 =|N\lambda(k_j)-\eta|\longrightarrow0.
\]
If \(\mathcal L-i\eta\) were surjective, its inverse would be bounded
by the closed graph theorem, contradicting this sequence. These points
therefore belong to the continuous spectrum.

A single Fourier mode rules out a uniform
\(L^1\)--\(L^\infty\) estimate with a time factor tending to zero.
For \(k=(n,m)\in K\) with \(n\ne0\), the real linear wave
\[
 v(t)=\frac{\cos(\xi_k\cdot(x,z)+N\lambda(k)t)}{\sqrt\gamma}
 \quad\text{and}\quad
 \theta(t)=\frac{\sin(\xi_k\cdot(x,z)+N\lambda(k)t)}{\sqrt{g\beta}}
\]
has finite initial \(L^1\) norms and satisfies
\(\|v(t)\|_{L^\infty}+\|\theta(t)\|_{L^\infty}
=\gamma^{-\frac12}+(g\beta)^{-\frac12}\) for every \(t\).
Such decay is therefore impossible even for smooth data with zero
vertical means.

For the quadratic interactions, a convolution triple \(k=k_1+k_2\)
is called admissible if \(k_j=(n_j,m_j)\in K\) and \(k=(n,m)\in K\),
equivalently \(m_1m_2m\ne0\). Define its phase by
\begin{equation}
\label{eq:2.1.36}
 \Phi_{\sigma\mu\nu}(k;k_1,k_2)
 :=\omega_\sigma(k)-\omega_\mu(k_1)-\omega_\nu(k_2),
 \qquad \sigma,\mu,\nu\in\{\pm1\}.
\end{equation}
The Fourier symbol of \(b\) is
\begin{equation}
\label{eq:2.1.37}
 m_b(k_1,k_2)
 :=i\kappa_x\left(\frac{n_1m_2}{m_1}-n_2\right)
 =i\kappa_x\frac{n_1m_2-n_2m_1}{m_1}.
\end{equation}
For a fixed convolution triple, set
\[
        m_{12}:=m_b(k_1,k_2),
        \qquad
        m_{21}:=m_b(k_2,k_1).
\]
For output branch \(\sigma\) and ordered input branches \((\mu,\nu)\),
the polarized wave coefficient of
\(\mathcal N_2(H_1,H_2)+\mathcal N_2(H_2,H_1)\) is
\begin{equation}
\label{eq:2.1.38}
 \widetilde q_{\sigma\mu\nu}(k;k_1,k_2)
 =\frac{
 \mathbf 1_{\{\nu=\sigma\}}m_{12}
 +\mathbf 1_{\{\mu=\sigma\}}m_{21}}
 {2\sqrt\gamma}.
\end{equation}
Section~\ref{sec:3-quadratic-wave-interactions} derives this formula from
\eqref{eq:2.1.12}. The coefficient includes both input orders without an
additional factor \(\frac12\).

An admissible interaction is resonant if \(\Phi_{\sigma\mu\nu}=0\).
A resonance is quadratically active when
\(\widetilde q_{\sigma\mu\nu}\ne0\) and quadratically transparent
when \(\widetilde q_{\sigma\mu\nu}=0\). We call a resonance nontrivial
if \((n_1,n_2)\ne(0,0)\).

\subsection{Main results}
\label{subsec:2.2-main-results}

The results below describe the sharp derivative loss of the second
Picard operator, classify exact resonances and transparent Fourier
supports, and establish uniform horizontal-mean bounds and resonant
averaging. All proofs are given in
Section~\ref{sec:4-exact-resonances-picard-obstruction}.

For \(H_1,H_2\in X^\infty\) and \(t\in\mathbb R\), define the
polarized second Picard operator by
\begin{equation}
\label{eq:2.2.1}
 \mathcal B_t(H_1,H_2)
 :={}\int_0^t e^{(t-\tau)\mathcal L}
 \Bigl[
 \mathcal N_2\bigl(e^{\tau\mathcal L}H_1,
                  e^{\tau\mathcal L}H_2\bigr)
 +\mathcal N_2\bigl(e^{\tau\mathcal L}H_2,
                  e^{\tau\mathcal L}H_1\bigr)
 \Bigr] \,\mathrm d\tau.
\end{equation}

\Needspace{6\baselineskip}
The first theorem determines the sharp derivative loss of
\(\mathcal B_t\).

\begin{theorem}[Sharp second Picard obstruction]
\label{thm:2.2.1-second-picard-obstruction}
Let \(\mathcal B_t\) be defined by \eqref{eq:2.2.1}.
\begin{enumerate}[(i)]
\item Let \(t\in\mathbb R\setminus\{0\}\), \(s\in\mathbb R\), and
      \(0\le\delta<1\). Then \(\mathcal B_t\) has no bounded bilinear
      extension from \(X^s\times X^s\) to \(X^{s-\delta}\). More precisely,
      there exist real trigonometric polynomials
      \(H,G_M\in X^\infty\), \(M\ge2\), with \(H\) independent of \(M\),
      such that
      \[
       \|H\|_{X^s}
       +\sup_{M\ge2}\|G_M\|_{X^s}
       \le C,
      \]
      and, for all sufficiently large \(M\),
      \begin{equation}
      \label{eq:2.2.2}
       \bigl\|\mathcal B_t(H,G_M)\bigr\|_{X^{s-\delta}}
       \ge c|t|M^{2(1-\delta)}.
      \end{equation}
      Here \(c,C>0\) are independent of \(M\).

\item Let \(s\ge2\) and \(t\in\mathbb R\). Then \(\mathcal B_t\) extends
      uniquely to a bounded bilinear map from \(X^s\times X^s\) to
      \(X^{s-1}\), and
      \begin{equation}
      \label{eq:2.2.3}
       \|\mathcal B_t(H_1,H_2)\|_{X^{s-1}}
       \le C_s|t|\,
       \|H_1\|_{X^s}\|H_2\|_{X^s}.
      \end{equation}
\end{enumerate}
Thus, for every \(s\ge2\) and fixed \(t\ne0\), the loss of one Sobolev
derivative is sharp.
\end{theorem}

We next classify the exact resonances of \eqref{eq:2.1.36} and compute
their polarized coefficients.
For each ordered pair \(m_1,m_2\in\mathbb Z_*\) with
\(m=m_1+m_2\ne0\), set
\[
 d_{\mathrm s}:=\gcd(m_1^2,m_2^2)
 \quad\text{and}\quad
 d_+:=\gcd(m_1^2,|m_2(m_2+2m_1)|),
\]
\[
 d_-:=\gcd(m_2^2,|m_1(m_1+2m_2)|)
\]
and
\[
 d_{\mathrm o}:=\gcd\bigl(|m_2(2m_1+m_2)|,|m_1(m_1+2m_2)|\bigr).
\]

\begin{theorem}[Exact lattice resonances]
\label{thm:2.2.2-exact-lattice-resonances}
Fix \(\sigma\in\{\pm1\}\) and an ordered pair
\(m_1,m_2\in\mathbb Z_*\) with \(m=m_1+m_2\ne0\).
In each channel, an admissible triple is resonant if and only if its
horizontal indices \((n_1,n_2,n)\) have the form below for a unique
\(\ell\in\mathbb Z\). The corresponding polarized coefficients are
as follows.
\begin{enumerate}[(i)]
\item \emph{Same-branch interactions:}
\((\mu,\nu)=(\sigma,\sigma)\). The phase relation is
\begin{equation}
 \Phi_{\sigma\sigma\sigma}(k;k_1,k_2)=0
 \quad\Longleftrightarrow\quad
 n_1m_2^2+n_2m_1^2=0,
\end{equation}
and the parametrization is
\begin{equation}
\label{eq:2.2.5}
 (n_1,n_2,n)
 =\ell\left(
 \frac{m_1^2}{d_{\mathrm s}},
 -\frac{m_2^2}{d_{\mathrm s}},
 \frac{m_1^2-m_2^2}{d_{\mathrm s}}
 \right).
\end{equation}
On this resonant set,
\begin{equation}
\label{eq:2.2.6}
 \widetilde q_{\sigma\sigma\sigma}
 =\frac{i\kappa_x\ell}{2\sqrt\gamma}
  \frac{m_2^2-m_1^2}{d_{\mathrm s}}
 =-\frac{i\kappa_x n}{2\sqrt\gamma}.
\end{equation}
The interaction is active precisely when
\(\ell(m_2^2-m_1^2)\ne0\). Since \(m_1+m_2\ne0\), a nontrivial
same-branch resonance is transparent precisely when \(m_1=m_2\).

\item \emph{First mixed channel:}
\((\mu,\nu)=(-\sigma,\sigma)\). In this case,
\begin{equation}
\label{eq:2.2.7}
 (n_1,n_2,n)
 =\ell\left(
 \frac{m_1^2}{d_+},
 \frac{m_2(m_2+2m_1)}{d_+},
 \frac{m^2}{d_+}
 \right).
\end{equation}
On this set,
\begin{equation}
\label{eq:2.2.8}
 \widetilde q_{\sigma,-\sigma,\sigma}
 =-\frac{i\kappa_x\ell}{2\sqrt\gamma}
   \frac{m_2m}{d_+}
 =-\frac{i\kappa_x n}{2\sqrt\gamma}\frac{m_2}{m}.
\end{equation}
Hence every resonance with \(\ell\ne0\) is active.

\item \emph{Second mixed channel:}
\((\mu,\nu)=(\sigma,-\sigma)\). In this case,
\begin{equation}
\label{eq:2.2.9}
 (n_1,n_2,n)
 =\ell\left(
 \frac{m_1(m_1+2m_2)}{d_-},
 \frac{m_2^2}{d_-},
 \frac{m^2}{d_-}
 \right).
\end{equation}
On this set,
\begin{equation}
\label{eq:2.2.10}
 \widetilde q_{\sigma,\sigma,-\sigma}
 =-\frac{i\kappa_x\ell}{2\sqrt\gamma}
   \frac{m_1m}{d_-}
 =-\frac{i\kappa_x n}{2\sqrt\gamma}\frac{m_1}{m}.
\end{equation}
Hence every resonance with \(\ell\ne0\) is active.

\item \emph{Opposite-input channel:}
\((\mu,\nu)=(-\sigma,-\sigma)\). The parametrization is
\begin{equation}
\label{eq:2.2.11}
 (n_1,n_2,n)
 =\ell\left(
 \frac{m_1(m_1+2m_2)}{d_{\mathrm o}},
 -\frac{m_2(2m_1+m_2)}{d_{\mathrm o}},
 \frac{m_1^2-m_2^2}{d_{\mathrm o}}
 \right).
\end{equation}
For every admissible convolution triple, resonant or not,
\begin{equation}
        \widetilde q_{\sigma,-\sigma,-\sigma}=0.
\end{equation}
Thus this channel is quadratically transparent.
\end{enumerate}
\end{theorem}

\begin{remark}
The classification includes zero horizontal input and output frequencies,
and a resonance is nontrivial exactly when \(\ell\ne0\).
Exchanging \((k_1,\mu)\) and \((k_2,\nu)\) interchanges cases
\emph{(ii)} and \emph{(iii)} without changing \(\ell\), and preserves
cases \emph{(i)} and \emph{(iv)} with \(\ell\mapsto-\ell\).
Simultaneous reflection \(k_j\mapsto-k_j\), \(j=1,2\), and
\((\sigma,\mu,\nu)\mapsto(-\sigma,-\mu,-\nu)\) sends
\(\ell\) to \(-\ell\) in every case, consistently with
\eqref{eq:2.1.30}.
\end{remark}

To formulate the support classification, let
\(S=-S\subset K\) and \(s\in\mathbb R\), and define
\[
 \mathcal R_S^s
 :=\left\{U=(v,\theta)\in X^s:
 \operatorname{supp}\widehat v\cup
 \operatorname{supp}\widehat\theta\subset S\right\}.
\]
We call \(S\) nonlinearly transparent if
\begin{equation}
\label{eq:2.2.13}
        \mathcal N_2(U,U)=0,
        \qquad
        \mathcal N_3(U,U,U)=0
        \qquad
        \text{for every }U\in\mathcal R_S^\infty,
\end{equation}
where \(\mathcal R_S^\infty:=\bigcap_{r\ge0}\mathcal R_S^r\).
Since \(\mathcal R_S^\infty\) is linear and the two terms have different
homogeneities, this is equivalent to
\(\mathcal N_2(U,U)+\mathcal N_3(U,U,U)=0\) for all
\(U\in\mathcal R_S^\infty\).

A vector \(k_0=(n_0,m_0)\in K\) is primitive if
\(\gcd(|n_0|,|m_0|)=1\). For such a vector, set
\[
        \Gamma_{k_0}:=\{j k_0:j\in\mathbb Z_*\},
        \qquad
        \mathcal R_{k_0}^s:=\mathcal R_{\Gamma_{k_0}}^s.
\]
We call \(\Gamma_{k_0}\) a lattice ray, including both signs of the
generator to accommodate real-valued states. The next theorem
characterizes the transparent spaces defined solely by a common
Fourier-support constraint.

\begin{theorem}[Transparent Fourier supports]
\label{thm:2.2.3-maximal-transparent-rays}
Let \(\varnothing\ne S=-S\subset K\). The following statements are
equivalent.
\begin{enumerate}[(i)]
\item The set \(S\) is nonlinearly transparent in the sense of
      \eqref{eq:2.2.13}.
\item Every two frequencies in \(S\) are collinear:
      \[
             n_1m_2-n_2m_1=0
             \qquad
             \text{for all }k_j=(n_j,m_j)\in S.
      \]
\item \(S\subset\Gamma_{k_0}\) for some primitive \(k_0\in K\).
\end{enumerate}
Under these equivalent conditions, the stronger ordered identities
\begin{equation}
\label{eq:2.2.14}
 \mathcal N_2(U_1,U_2)=0,
 \qquad
 \mathcal N_3(U_1,U_2,U_3)=0
 \qquad
 \bigl(U_j\in\mathcal R_S^\infty\bigr)
\end{equation}
hold. For every \(s\in\mathbb R\), the nonzero nonlinearly transparent
subspaces of the form \(\mathcal R_S^s\) that are maximal under inclusion
are precisely the ray spaces \(\mathcal R_{k_0}^s\). Their primitive
generators are unique up to sign.
\end{theorem}

The final theorem identifies the large-time limit of the normalized
quadratic response after removing the linear evolution. For
\(H_1,H_2\in X^\infty\), set \(Z_{j,\mu}:=Z_\mu[H_j]\) and define the
resonant operator \(\mathcal Q_{\mathrm{res}}\) by
\begin{equation}
\label{eq:2.2.15}
 \widehat{Z_\sigma[\mathcal Q_{\mathrm{res}}(H_1,H_2)]}(k)
 =\sum_{\mu,\nu=\pm}
  \sum_{\substack{k_1+k_2=k\\\Phi_{\sigma\mu\nu}=0}}
  \widetilde q_{\sigma\mu\nu}
  \widehat Z_{1,\mu}(k_1)\widehat Z_{2,\nu}(k_2),
 \qquad k\in K.
\end{equation}
The inner sum runs over admissible triples. All wave coefficients with
zero vertical index are set to zero, and \eqref{eq:2.1.26} determines the
constrained output. As in \(\mathcal B_t\), the coefficient is polarized,
with no additional factor \(\frac12\).

\begin{theorem}[Horizontal means and resonant averaging]
\label{thm:2.2.4-mean-resonant-averaging}
Let \(\mathcal B_t\) and \(\mathcal Q_{\mathrm{res}}\) be defined by
\eqref{eq:2.2.1} and \eqref{eq:2.2.15}, respectively. Then the following
statements hold.
\begin{enumerate}[(i)]
\item Every admissible resonance with zero horizontal output is
      quadratically transparent:
      \begin{equation}
      \label{eq:2.2.16}
       n=0,\quad \Phi_{\sigma\mu\nu}=0
       \quad\Longrightarrow\quad \widetilde q_{\sigma\mu\nu}=0.
      \end{equation}
      For every \(s\ge2\) and \(H_1,H_2\in X^s\),
      \begin{equation}
      \label{eq:2.2.17}
       \sup_{t\in\mathbb R}
       \bigl\|\langle\mathcal B_t(H_1,H_2)\rangle_x\bigr\|_{X^{s-1}}
       \le C_s\|H_1\|_{X^s}\|H_2\|_{X^s}.
      \end{equation}
\item For every \(s\ge2\), \(\mathcal Q_{\mathrm{res}}\) extends
      uniquely to a bounded symmetric real-bilinear map from
      \(X^s\times X^s\) to \(X^{s-1}\). For all \(H_1,H_2\in X^s\),
      \begin{equation}
      \label{eq:2.2.18}
       \|\mathcal Q_{\mathrm{res}}(H_1,H_2)\|_{X^{s-1}}
       \le C_s\|H_1\|_{X^s}\|H_2\|_{X^s}.
      \end{equation}
      Its horizontal mean vanishes. For each fixed pair
      \(H_1,H_2\in X^s\),
      \begin{equation}
      \label{eq:2.2.19}
       \frac{1}{t} e^{-t\mathcal L}\mathcal B_t(H_1,H_2)
       \longrightarrow\mathcal Q_{\mathrm{res}}(H_1,H_2)
       \quad\text{in }X^{s-1}\quad\text{as }|t|\to\infty.
      \end{equation}
\item For every \(s\ge4\) and \(H_1,H_2\in X^s\),
      \begin{equation}
      \label{eq:2.2.20}
       \sup_{t\in\mathbb R}
       \bigl\|e^{-t\mathcal L}\mathcal B_t(H_1,H_2)
             -t\mathcal Q_{\mathrm{res}}(H_1,H_2)\bigr\|_{X^{s-4}}
       \le C_s\|H_1\|_{X^s}\|H_2\|_{X^s}.
      \end{equation}
      Consequently, for every \(t\ne0\),
      \begin{equation}
      \label{eq:2.2.21}
       \left\|\frac{1}{t} e^{-t\mathcal L}\mathcal B_t(H_1,H_2)
                   -\mathcal Q_{\mathrm{res}}(H_1,H_2)\right\|_{X^{s-4}}
       \le\frac{C_s}{|t|}\|H_1\|_{X^s}\|H_2\|_{X^s}.
      \end{equation}
\end{enumerate}
The constants depend only on \(s,L_x,L_z,g,\beta,\gamma\), not on time
or the inputs. The complex-bilinear extensions satisfy the same estimates
and convergence statements.
\end{theorem}

\begin{remark}
Convergence in \eqref{eq:2.2.19} is asserted for each fixed input pair,
whereas \eqref{eq:2.2.21} gives convergence in bilinear operator norm
into \(X^{s-4}\) for \(s\ge4\). The four-derivative loss is not claimed
to be optimal. These statements concern responses to free linear
trajectories, not convergence or resonant approximation of nonlinear
solutions.
\end{remark}



\section{Wave multipliers and multilinear estimates}
\label{sec:3-quadratic-wave-interactions}

We derive the wave multipliers and the Fourier representation of
\(\mathcal B_t\). We then establish multilinear Sobolev estimates,
including the endpoint \(s=2\), and the phase-division bounds used
in resonant averaging.

\subsection{Wave interactions}
\label{subsec:3.1-wave-multipliers}

Following the conventions of
Section~\ref{sec:2-preliminaries-main-results}, we write
\(k_j=(n_j,m_j)\in K\), \(j=1,2\), and \(k=k_1+k_2=(n,m)\)
in convolution sums. The horizontal indices may vanish, and the
vertical input indices may have either sign. Since \(\mathcal P\)
removes quadratic outputs with \(m=0\), output branches and phases
are considered only for \(m\ne0\). Frequency arguments are omitted
when clear from the context.

For \(U_j=(v_j,\theta_j)\in X^\infty_{\mathbb C}\), write
\(Z_{j,\sigma}=Z_\sigma[U_j]\) as in \eqref{eq:2.1.25}.
By the bilinearity of \(b\), the definition \eqref{eq:2.1.7}
and the inverse transformation \eqref{eq:2.1.26} imply
\begin{align}
 Z_\sigma[\mathcal N_2(U_1,U_2)]
 ={}&\mathcal P\Bigl[
      \sqrt\gamma\,b(v_1,v_2)
      +i\sigma\sqrt{g\beta}\,b(v_1,\theta_2)
      \Bigr] \notag\\
 ={}&\mathcal P b\bigl(v_1,Z_{2,\sigma}\bigr)
 =\frac{1}{2\sqrt\gamma}
   \sum_{\mu=\pm}
   \mathcal P b\bigl(Z_{1,\mu},Z_{2,\sigma}\bigr).
\label{eq:3.1}
\end{align}
The cubic term in \eqref{eq:2.1.8} takes the form
\begin{equation}
 Z_\sigma[\mathcal N_3(U_1,U_2,U_3)]
 =\frac{1}{\gamma}
   \partial_x\overline{v_1\theta_2}\,
   \partial_zZ_{3,\sigma}.
\end{equation}
Combining these identities with \eqref{eq:2.1.28}, we write
\eqref{eq:2.1.11} for a smooth solution \(U\) as
\begin{equation}
\label{eq:3.3}
 \partial_tZ_\sigma
 =i\sigma N\partial_z^{-1}\partial_xZ_\sigma
 +\frac{1}{2\sqrt\gamma}
  \sum_{\mu=\pm}\mathcal P b(Z_\mu,Z_\sigma)
 +\frac{1}{\gamma}
   \partial_x\overline{v\theta}\,\partial_zZ_\sigma.
\end{equation}
To remove the linear evolution, define the profile
\begin{equation}
\label{eq:3.4}
 F_\sigma(t):=e^{-it\sigma N\partial_z^{-1}\partial_x}Z_\sigma(t).
\end{equation}
For real-valued \(U\), the identities
\(\omega_\sigma(-k)=\omega_\sigma(k)=-\omega_{-\sigma}(k)\)
and \eqref{eq:2.1.30} imply
\begin{equation}
\label{eq:3.5}
 \widehat F_\sigma(t,-k)
 =e^{-it\omega_\sigma(k)}
   \overline{\widehat Z_{-\sigma}(t,k)}
 =\overline{\widehat F_{-\sigma}(t,k)}.
\end{equation}

We next express the quadratic term in Fourier variables.
For smooth \(f,g\) with zero vertical mean, the symbol of \(b\) is
\(\lambda(k_1)i\kappa_zm_2-i\kappa_xn_2=m_{12}\), as in
\eqref{eq:2.1.37}. If \(m_1+m_2=0\), this equals
\(-i\kappa_x(n_1+n_2)\) and cancels the coefficient
\(i\kappa_x(n_1+n_2)\) of
\(\partial_x\overline{fg}\) in \eqref{eq:2.1.5}. Thus
\begin{equation}
\label{eq:3.6}
 \widehat{\mathcal P b(f,g)}(k)
 =
 \begin{cases}
 \displaystyle
 \sum_{k_1+k_2=k}
 m_{12}\widehat f(k_1)\widehat g(k_2),
 &m\ne0,\\[2mm]
 0,&m=0.
 \end{cases}
\end{equation}
For \(k\in K\), define the ordered branch multipliers
\(q_{\sigma\mu\nu}\) by
\begin{equation}
\label{eq:3.7}
 \widehat{Z_\sigma[\mathcal N_2(U_1,U_2)]}(k)
 =\sum_{\mu,\nu=\pm}
  \sum_{k_1+k_2=k}
  q_{\sigma\mu\nu}(k;k_1,k_2)
  \widehat Z_{1,\mu}(k_1)\widehat Z_{2,\nu}(k_2).
\end{equation}
The following lemma identifies the ordered and polarized coefficients.

\begin{lemma}
\label{lem:3.1-quadratic-branch-multipliers}
Let \(k_1,k_2\in K\) with \(k=k_1+k_2\in K\), and let
\(\sigma,\mu,\nu\in\{\pm1\}\). Then
\begin{equation}
\label{eq:3.8}
 q_{\sigma\mu\nu}(k;k_1,k_2)
 =\frac{\mathbf 1_{\{\nu=\sigma\}}}{2\sqrt\gamma}\,m_{12}.
\end{equation}
The polarized coefficient \(\widetilde q_{\sigma\mu\nu}\) is given
by \eqref{eq:2.1.38}.
\end{lemma}

\begin{proof}
Combining \eqref{eq:3.1} with \eqref{eq:3.6} and comparing the
result with \eqref{eq:3.7} establishes \eqref{eq:3.8}.
For \(H_1,H_2\in X^\infty_{\mathbb C}\), write
\(Z_{j,\rho}=Z_\rho[H_j]\). In the reversed term
\(\mathcal N_2(H_2,H_1)\), exchanging \((k_1,\mu)\) and
\((k_2,\nu)\) identifies the coefficient of
\(\widehat Z_{1,\mu}(k_1)\widehat Z_{2,\nu}(k_2)\) as
\(q_{\sigma\nu\mu}(k;k_2,k_1)\). Hence
\begin{equation}
\label{eq:3.9}
 \widetilde q_{\sigma\mu\nu}(k;k_1,k_2)
 =q_{\sigma\mu\nu}(k;k_1,k_2)
  +q_{\sigma\nu\mu}(k;k_2,k_1).
\end{equation}
Substituting \eqref{eq:3.8} establishes \eqref{eq:2.1.38}.
The polarized quadratic term therefore has the representation
\begin{equation}
\label{eq:3.10}
 \widehat{Z_\sigma[\mathcal N_2(H_1,H_2)+\mathcal N_2(H_2,H_1)]}(k)
 =\sum_{\mu,\nu=\pm}\sum_{k_1+k_2=k}
 \widetilde q_{\sigma\mu\nu}
 \widehat Z_{1,\mu}(k_1)\widehat Z_{2,\nu}(k_2).
\end{equation}
\end{proof}

Clearing the vertical denominators in \eqref{eq:2.1.36} gives,
for an admissible triple,
\begin{equation}
\label{eq:3.11}
 \frac{L_x}{L_z}\,m_1m_2m
 \frac{\Phi_{\sigma\mu\nu}(k;k_1,k_2)}{\sigma N}
 =(n_1+n_2)m_1m_2
 -\frac{\mu}{\sigma}n_1m_2m
 -\frac{\nu}{\sigma}n_2m_1m.
\end{equation}
Since \(m_1m_2m\ne0\), resonance is equivalent to the vanishing of
the right-hand side. For fixed vertical indices and branch signs,
this is a linear Diophantine equation in \((n_1,n_2)\).
The phase and polarized coefficient are invariant under exchange
of the input pairs:
\[
 \Phi_{\sigma\mu\nu}(k;k_1,k_2)
 =\Phi_{\sigma\nu\mu}(k;k_2,k_1),
\]
and
\begin{equation}
\label{eq:3.12}
 \widetilde q_{\sigma\mu\nu}(k;k_1,k_2)
 =\widetilde q_{\sigma\nu\mu}(k;k_2,k_1).
\end{equation}
For the same-branch interaction, substitution of
\eqref{eq:2.1.37} into \eqref{eq:2.1.38} gives the factorization
\begin{equation}
 \widetilde q_{\sigma\sigma\sigma}(k;k_1,k_2)
 =\frac{m_{12}+m_{21}}{2\sqrt\gamma}
 =\frac{i\kappa_x}{2\sqrt\gamma}
 (n_1m_2-n_2m_1)
 \left(\frac{1}{m_1}-\frac{1}{m_2}\right).
\label{eq:3.13}
\end{equation}

We now write \eqref{eq:3.3} in terms of the profiles
\eqref{eq:3.4}. Symmetrization by \eqref{eq:3.9} yields, for
\(k\in K\),
\begin{equation}
 \partial_t\widehat F_\sigma(t,k)
 =\frac{1}{2}\sum_{\mu,\nu=\pm}\sum_{k_1+k_2=k}
 e^{-it\Phi_{\sigma\mu\nu}}\widetilde q_{\sigma\mu\nu}
 \widehat F_\mu(t,k_1)\widehat F_\nu(t,k_2)
 +\widehat{\mathcal R}_\sigma(t,k),
\end{equation}
where the cubic term is
\(\mathcal R_\sigma(t):=
 e^{-it\sigma N\partial_z^{-1}\partial_x}
 Z_\sigma[\mathcal N_3(U,U,U)(t)]\).
The factor \(\frac12\) averages the two input orders in the
self-interaction. The polarized second differential
\eqref{eq:2.1.12} contains their sum, without this factor.

For \(H_1,H_2\in X^\infty\), set \(Z_{j,\mu}=Z_\mu[H_j]\).
Substituting \eqref{eq:3.10} into the definition
\eqref{eq:2.2.1} and using
Proposition~\ref{prop:2.1.1-unitary-linear-evolution} with
\eqref{eq:2.1.28}, we obtain the Fourier--Duhamel representation
\begin{equation}
\label{eq:3.15}
\begin{aligned}
 \widehat{Z_\sigma[\mathcal B_t(H_1,H_2)]}(k)
 ={}&e^{it\omega_\sigma(k)}
 \sum_{\mu,\nu=\pm}\sum_{k_1+k_2=k}
 \widetilde q_{\sigma\mu\nu}\widehat Z_{1,\mu}(k_1)\widehat Z_{2,\nu}(k_2)\\
 &\quad\times\int_0^t e^{-i\tau\Phi_{\sigma\mu\nu}}\,\mathrm d\tau
\end{aligned}
\end{equation}
for \(t\in\mathbb R\) and \(k\in K\).
Rapid Fourier decay and polynomial growth of the multiplier justify
the interchange of summation and integration on compact time intervals.
For real \(\Phi\),
\[
 \int_0^t e^{-i\tau\Phi}\,\mathrm d\tau
 =
 \begin{cases}
 t,&\Phi=0,\\[1mm]
 \displaystyle\frac{1-e^{-it\Phi}}{i\Phi},&\Phi\ne0.
 \end{cases}
\]
The identities \(m_b(-k_1,-k_2)=\overline{m_b(k_1,k_2)}\) and
\(\omega_{-\sigma}(-k)=-\omega_\sigma(k)\) imply
\[
 \Phi_{-\sigma,-\mu,-\nu}(-k;-k_1,-k_2)
 =-\Phi_{\sigma\mu\nu}(k;k_1,k_2),
\]
and
\begin{equation}
\label{eq:3.16}
 \widetilde q_{-\sigma,-\mu,-\nu}(-k;-k_1,-k_2)
 =\overline{\widetilde q_{\sigma\mu\nu}(k;k_1,k_2)}.
\end{equation}
For real-valued \(U\), the cubic term satisfies
\(\widehat{\mathcal R}_\sigma(t,-k)
=\overline{\widehat{\mathcal R}_{-\sigma}(t,k)}\).
Together, these relations show that the quadratic and cubic profile
terms preserve \eqref{eq:3.5}.

\subsection{Multilinear estimates}
\label{subsec:3.2-multilinear-endpoint}
\label{subsec:3.3-phase-division}

We first estimate the quadratic and cubic terms, using mixed norms
at \(s=2\). The remaining lemmas control division by nonzero phases
and the resulting bilinear Fourier sums.

\begin{proposition}
\label{prop:3.1-quadratic-endpoint-bounds}
Let \(s\ge2\). The operators \(\mathcal N_2\) and \(\mathcal N_3\)
admit unique bounded bilinear and trilinear extensions
\[
 \mathcal N_2:X^s\times X^s\longrightarrow X^{s-1}
 \quad\text{and}\quad
 \mathcal N_3:(X^s)^3\longrightarrow X^{s-1},
\]
respectively. For all \(U_1,U_2,U_3\in X^s\),
\begin{equation}
\label{eq:3.17}
 \|\mathcal N_2(U_1,U_2)\|_{X^{s-1}}
 \le C_s\|U_1\|_{X^s}\|U_2\|_{X^s},
\end{equation}
and
\begin{equation}
\label{eq:3.18}
 \|\mathcal N_3(U_1,U_2,U_3)\|_{X^{s-1}}
 \le C_s\|U_1\|_{X^s}\|U_2\|_{X^s}\|U_3\|_{X^s}.
\end{equation}
The same conclusions hold on the corresponding complexified spaces.
\end{proposition}

\begin{proof}
We first take smooth inputs. Suppose that \(s>2\) and set \(r:=s-1>1\).
For \(f,g\) with zero vertical mean, the algebra property of
\(H^r(\Omega)\) and \eqref{eq:2.1.17} imply
\begin{align*}
 \|b(f,g)\|_{H^r}
 \le{}&C_r\Bigl(
 \|\partial_z^{-1}\partial_xf\|_{H^r}\|\partial_zg\|_{H^r}
 +\|f\|_{H^r}\|\partial_xg\|_{H^r}\Bigr)\\
 \le{}&C_s\|f\|_{H^s}\|g\|_{H^s}.
\end{align*}
Applying this estimate componentwise in \eqref{eq:2.1.7} and using
the boundedness of \(\mathcal P\) on \(H^r\) proves \eqref{eq:3.17}.

For the cubic term, regard \(\overline{v_1\theta_2}\) as a function
on \(\Omega\) independent of \(z\). The boundedness of vertical
averaging and the algebra property of \(H^s(\Omega)\) imply
\[
 \|\partial_x\overline{v_1\theta_2}\|_{H^r}
 \le C_s\|v_1\theta_2\|_{H^s}
 \le C_s\|v_1\|_{H^s}\|\theta_2\|_{H^s}.
\]
Applying the \(H^r\) product estimate to \eqref{eq:2.1.8}, we obtain
\begin{align*}
 \|\mathcal N_3(U_1,U_2,U_3)\|_{X^r}
 \le{}&C_s\|\partial_x\overline{v_1\theta_2}\|_{H^r}
 \bigl(\|\partial_zv_3\|_{H^r}+\|\partial_z\theta_3\|_{H^r}\bigr)\\
 \le{}&C_s\|U_1\|_{X^s}\|U_2\|_{X^s}\|U_3\|_{X^s}.
\end{align*}
This proves \eqref{eq:3.18} for \(s>2\).

At \(s=2\), let \(a:=\partial_z^{-1}\partial_xf\).
The periodic inverse derivative maps zero-mean \(L_z^2\) functions
boundedly into \(H_z^1\hookrightarrow L_z^\infty\). Together with
\(H_x^1(L_z^2)\hookrightarrow L_x^\infty L_z^2\), this gives
\begin{equation}
\label{eq:3.19}
\begin{aligned}
 &\|a\|_{L^\infty}\le C\|\partial_xf\|_{H_x^1(L_z^2)}\le C\|f\|_{H^2},\\
 &\|\partial_xa\|_{L_x^2L_z^\infty}\le C\|\partial_{xx}f\|_{L^2}\le C\|f\|_{H^2}.
\end{aligned}
\end{equation}
The one-dimensional embedding in \(x\) and
\(H^1(\Omega)\hookrightarrow L^4(\Omega)\) also imply
\[
 \|\partial_zg\|_{L_x^\infty L_z^2}
 +\|\nabla g\|_{L^4}\le C\|g\|_{H^2}.
\]
Using \(\partial_za=\partial_xf\) in the product rule and applying
H\"older's inequality, we estimate
\begin{align*}
 \|a\partial_zg\|_{H^1}
 \le{}&C\|a\|_{L^\infty}\|g\|_{H^2}
       +C\|\partial_xa\|_{L_x^2L_z^\infty}
          \|\partial_zg\|_{L_x^\infty L_z^2}\\
 &+C\|\partial_xf\|_{L^4}\|\partial_zg\|_{L^4}\\
 \le{}&C\|f\|_{H^2}\|g\|_{H^2}.
\end{align*}
The other transport term satisfies
\[
 \|f\partial_xg\|_{H^1}
 \le C\bigl(\|f\|_{L^\infty}\|g\|_{H^2}
            +\|\nabla f\|_{L^4}\|\partial_xg\|_{L^4}\bigr)
 \le C\|f\|_{H^2}\|g\|_{H^2}.
\]
Thus \(\|b(f,g)\|_{H^1}\le C\|f\|_{H^2}\|g\|_{H^2}\).
Applying the bounded projection \(\mathcal P\) proves
\eqref{eq:3.17} at \(s=2\).

For the cubic term, set
\(A(x):=\frac1\gamma\partial_x\overline{v_1\theta_2}(x)\).
The algebra property of \(H^2(\Omega)\) and boundedness of vertical
averaging ensure that
\[
 \|A\|_{H_x^1}\le C\|v_1\theta_2\|_{H^2}
 \le C\|v_1\|_{H^2}\|\theta_2\|_{H^2}.
\]
As \(A\) is independent of \(z\),
\begin{align*}
 \|A\partial_zg\|_{H^1}
 \le{}&C\|A\|_{L_x^\infty}\|g\|_{H^2}
       +C\|\partial_xA\|_{L_x^2}
          \|\partial_zg\|_{L_x^\infty L_z^2}\\
 \le{}&C\|A\|_{H_x^1}\|g\|_{H^2}.
\end{align*}
Taking \(g=v_3,\theta_3\) establishes \eqref{eq:3.18} at \(s=2\).

The quadratic output has zero vertical mean by projection. The cubic
output has the same property because its prefactor is independent of \(z\)
and \(\overline{\partial_zv_3}=\overline{\partial_z\theta_3}=0\).
Density then yields the unique extensions into \(X^{s-1}\).
The same estimates hold for complex inputs.
\end{proof}

The nonresonant part of the quadratic response involves division by
the phase. The integer numerator in \eqref{eq:3.11} provides a lower
bound for every nonzero phase, with no restriction on the aspect
ratio of the torus.

\begin{lemma}
\label{lem:3.2-phase-division}
For every admissible interaction with \(\Phi_{\sigma\mu\nu}\ne0\),
\begin{equation}
\label{eq:3.20}
 |\Phi_{\sigma\mu\nu}|
 \ge\frac{NL_z}{L_x|m_1m_2m|},
\end{equation}
and
\begin{equation}
\label{eq:3.21}
 \left|\frac{\widetilde q_{\sigma\mu\nu}}{\Phi_{\sigma\mu\nu}}\right|
 \le C|n_1m_2-n_2m_1|\,|m|\,(|m_1|+|m_2|)
 \le C\bigl(\langle\xi_{k_1}\rangle
              +\langle\xi_{k_2}\rangle\bigr)^4.
\end{equation}
The constants depend only on the fixed parameters.
\end{lemma}

\begin{proof}
The right-hand side of \eqref{eq:3.11} is an integer, since
\(\frac\mu\sigma,\frac\nu\sigma\in\{\pm1\}\). When the phase is
nonzero, this integer has absolute value at least one, which proves
\eqref{eq:3.20}. Moreover, \eqref{eq:2.1.37}--\eqref{eq:2.1.38} imply
\[
 |\widetilde q_{\sigma\mu\nu}|
 \le\frac{\kappa_x}{2\sqrt\gamma}|n_1m_2-n_2m_1|
       \left(\frac1{|m_1|}+\frac1{|m_2|}\right).
\]
Combining this estimate with \eqref{eq:3.20} cancels the input
denominators and proves the first inequality in \eqref{eq:3.21}.
The second follows from
\(|n_j|+|m_j|\le C\langle\xi_{k_j}\rangle\) and
\(|m|\le|m_1|+|m_2|\).
\end{proof}

To convert these pointwise multiplier bounds into Sobolev estimates,
we use the following convolution inequality. The argument also allows
restrictions to selected input pairs or output modes.

\begin{lemma}
\label{lem:3.3-fourier-summation}
Let \(d\ge0\), \(s>1\), and \(s\ge d\). Suppose that a scalar
multiplier \(a(k;k_1,k_2)\), with \(k=k_1+k_2\), satisfies
\[
 |a(k;k_1,k_2)|\le C_0
       \bigl(\langle\xi_{k_1}\rangle+\langle\xi_{k_2}\rangle\bigr)^d.
\]
For trigonometric polynomials, define
\[
 \widehat{\mathcal T(f,g)}(k)
 :=\sum_{k_1+k_2=k}a(k;k_1,k_2)\widehat f(k_1)\widehat g(k_2).
\]
Then \(\mathcal T\) extends uniquely to a bounded bilinear map from
\(H^s(\Omega;\mathbb C)\times H^s(\Omega;\mathbb C)\) to
\(H^{s-d}(\Omega;\mathbb C)\), with
\begin{equation}
\label{eq:3.22}
 \|\mathcal T(f,g)\|_{H^{s-d}}
 \le C_{s,d}C_0\|f\|_{H^s}\|g\|_{H^s}.
\end{equation}
\end{lemma}

\begin{proof}
Since \(s-d\ge0\) and \(\xi_k=\xi_{k_1}+\xi_{k_2}\),
\[
 \langle\xi_k\rangle^{s-d}|a(k;k_1,k_2)|
 \le C_{s,d}C_0
       \bigl(\langle\xi_{k_1}\rangle^s
             +\langle\xi_{k_2}\rangle^s\bigr).
\]
Taking absolute values term by term and applying Young's convolution
inequality on \(\mathbb Z^2\), we obtain
\begin{equation*}
 \|\mathcal T(f,g)\|_{H^{s-d}}
 \le C_{s,d}C_0\Bigl(
       \|\langle\xi_k\rangle^s\widehat f(k)\|_{\ell^2}
       \|\widehat g\|_{\ell^1}
       +\|\widehat f\|_{\ell^1}
        \|\langle\xi_k\rangle^s\widehat g(k)\|_{\ell^2}\Bigr).
\end{equation*}
The fixed volume factor is included in the constant. Since \(s>1\),
Cauchy--Schwarz shows that
\[
 \|\widehat f\|_{\ell^1}
 \le\left(\sum_{k\in\mathbb Z^2}\langle\xi_k\rangle^{-2s}\right)^{\frac12}
       \|\langle\xi_k\rangle^s\widehat f(k)\|_{\ell^2}
 \le C_s\|f\|_{H^s},
\]
and the same bound holds for \(g\). Substitution proves
\eqref{eq:3.22}, and density establishes the extension. Finally,
restricting the multiplier to selected input pairs or output modes
preserves its pointwise bound, so the estimate remains valid.
\end{proof}



\section{Proofs and refinements of the main results}
\label{sec:4-exact-resonances-picard-obstruction}

Subsection~\ref{subsec:4.1-main-proofs} proves
Theorems~\ref{thm:2.2.1-second-picard-obstruction}--\ref{thm:2.2.4-mean-resonant-averaging}
in the order stated, using the wave formulas and multilinear estimates
of Section~\ref{sec:3-quadratic-wave-interactions}.
Subsection~\ref{subsec:4.2-anisotropic-dyadic} then refines the Picard
bounds in tensor-product Sobolev spaces and on physical dyadic shells.

\subsection{Proofs of Theorems~\ref{thm:2.2.1-second-picard-obstruction}--\ref{thm:2.2.4-mean-resonant-averaging}}
\label{subsec:4.1-main-proofs}
\label{subsec:4.3-mean-averaging}

We begin with a same-branch resonant family that produces the lower
bound in Theorem~\ref{thm:2.2.1-second-picard-obstruction}.

\begin{corollary}
\label{cor:4.1.1-active-high-low-resonances}
For every \(\sigma\in\{\pm1\}\) and integer \(M\ge2\), the triple
\begin{equation}
\label{eq:4.1.1}
 (k_1,k_2,k)=\bigl((1,1),\,(-M^2,M),\,(1-M^2,M+1)\bigr)
\end{equation}
satisfies \(k=k_1+k_2\) and is an active exact resonance in the channel
\((\sigma,\sigma)\to\sigma\), with
\begin{equation}
 \widetilde q_{\sigma\sigma\sigma}(k;k_1,k_2)
 =\frac{i\kappa_x}{2\sqrt\gamma}(M^2-1).
\end{equation}
The frequencies and coefficient satisfy
\begin{equation}
\label{eq:4.1.3}
 \langle\xi_{k_2}\rangle\simeq\langle\xi_k\rangle
 \simeq|\widetilde q_{\sigma\sigma\sigma}(k;k_1,k_2)|\simeq M^2,
\end{equation}
uniformly for \(M\ge2\), with constants depending only on
\(L_x,L_z,\gamma\).
\end{corollary}

\begin{proof}
The vertical indices are nonzero, and the phase in \eqref{eq:2.1.36}
satisfies
\[
 \frac{\Phi_{\sigma\sigma\sigma}(k;k_1,k_2)}{\sigma N}
 =\frac{L_z}{L_x}\left(\frac{1-M^2}{M+1}-1+M\right)=0.
\]
By \eqref{eq:2.1.37}, the ordered multipliers are
\[
 m_{12}=i\kappa_x(M+M^2)
 \quad\text{and}\quad
 m_{21}=-i\kappa_x(M+1).
\]
The diagnostic contributions \(i\kappa_xM\) and \(-i\kappa_xM\)
cancel. Substitution into \eqref{eq:2.1.38} identifies the stated
coefficient, which is nonzero for \(M\ge2\). The frequency comparisons
follow from
\[
 |\xi_{k_2}|^2=\kappa_x^2M^4+\kappa_z^2M^2
 \quad\text{and}\quad
 |\xi_k|^2=\kappa_x^2(M^2-1)^2+\kappa_z^2(M+1)^2.
\]
The constants are uniform for \(M\ge2\).
\end{proof}

For \(k\in K\), \(a\in\mathbb R\), and \(\sigma\in\{\pm1\}\),
define the real-valued mode
\begin{equation}
\label{eq:4.1.4}
 E_\sigma(k;a)(x,z):=a
 \begin{pmatrix}
 \frac{1}{\sqrt\gamma}\cos(\xi_k\cdot(x,z))\\[2mm]
 \frac{\sigma}{\sqrt{g\beta}}\sin(\xi_k\cdot(x,z))
 \end{pmatrix}.
\end{equation}
By \eqref{eq:2.1.25} and \eqref{eq:2.1.29},
\begin{equation}
\label{eq:4.1.5}
 Z_\rho[E_\sigma(k;a)]=ae^{i\rho\sigma\xi_k\cdot(x,z)},
 \qquad \rho\in\{\pm1\},
\end{equation}
and
\begin{equation}
\label{eq:4.1.6}
 \|E_\sigma(k;a)\|_{X^r}^2=L_xL_z a^2\langle\xi_k\rangle^{2r},
 \qquad r\in\mathbb R.
\end{equation}
These modes belong to \(X^\infty\).

\begin{proof}[Proof of Theorem~\ref{thm:2.2.1-second-picard-obstruction}]
Fix \(t\ne0\), \(s\in\mathbb R\), and \(0\le\delta<1\).
For the frequencies in \eqref{eq:4.1.1}, choose the amplitudes
\begin{equation}
\label{eq:4.1.7}
 A_1:=(L_xL_z)^{-\frac12}\langle\xi_{k_1}\rangle^{-s},
\end{equation}
and
\begin{equation}
\label{eq:4.1.8}
 A_M:=(L_xL_z)^{-\frac12}\langle\xi_{k_2}\rangle^{-s}.
\end{equation}
Define the corresponding real inputs by
\begin{equation}
\label{eq:4.1.9}
 H:=E_+(k_1;A_1),
\end{equation}
and
\begin{equation}
\label{eq:4.1.10}
 G_M:=E_+(k_2;A_M).
\end{equation}
The normalization follows from \eqref{eq:4.1.6}:
\begin{equation}
\label{eq:4.1.11}
 \|H\|_{X^s}=\|G_M\|_{X^s}=1,\qquad M\ge2.
\end{equation}
By \eqref{eq:4.1.5}, the inputs occupy \(k_1,k_2\) in the \(+\)
branch and \(-k_1,-k_2\) in the \(-\) branch. The possible output
frequencies are therefore
\[
 \{\pm(k_1+k_2),\,\pm(k_1-k_2)\}
 =\{\pm k,\,\pm(M^2+1,1-M)\}.
\]
These frequencies are pairwise distinct for \(M\ge2\). At \(k\),
only the \((+,+)\) input branches contribute. Formula \eqref{eq:2.1.38}
implies \(\widetilde q_{-++}=0\), while \(\widetilde q_{+++}\)
includes both input orders.

By \eqref{eq:3.15} and
Corollary~\ref{cor:4.1.1-active-high-low-resonances},
\[
 \widehat{Z_+[\mathcal B_t(H,G_M)]}(k)
 =t e^{it\omega_+(k)}\widetilde q_{+++}(k;k_1,k_2)A_1A_M.
\]
In particular,
\begin{equation}
\label{eq:4.1.12}
 \left|\widehat{Z_+[\mathcal B_t(H,G_M)]}(k)\right|
 =\frac{\kappa_x}{2\sqrt\gamma}|t|(M^2-1)A_1A_M.
\end{equation}
The output is real-valued, so the coefficient at \((-k,-)\) is the
complex conjugate of the one at \((k,+)\). Retaining these two modes
in \eqref{eq:2.1.29}, we obtain
\[
 \|\mathcal B_t(H,G_M)\|_{X^{s-\delta}}^2
 \ge L_xL_z\langle\xi_k\rangle^{2(s-\delta)}
 \left|\widehat{Z_+[\mathcal B_t(H,G_M)]}(k)\right|^2.
\]
The amplitudes \eqref{eq:4.1.7}--\eqref{eq:4.1.8} and the coefficient
identity \eqref{eq:4.1.12} therefore imply
\begin{equation}
\label{eq:4.1.13}
 \|\mathcal B_t(H,G_M)\|_{X^{s-\delta}}
 \ge\frac{\kappa_x|t|}{2\sqrt{\gamma L_xL_z}}(M^2-1)
 \frac{\langle\xi_k\rangle^{s-\delta}}
      {\langle\xi_{k_1}\rangle^s\langle\xi_{k_2}\rangle^s}.
\end{equation}
The ratio \(\frac{\langle\xi_k\rangle}{\langle\xi_{k_2}\rangle}\)
is bounded above and away from zero by \eqref{eq:4.1.3}. Hence, for
every \(s\in\mathbb R\),
\[
 \frac{\langle\xi_k\rangle^{s-\delta}}
      {\langle\xi_{k_1}\rangle^s\langle\xi_{k_2}\rangle^s}
 =\langle\xi_{k_1}\rangle^{-s}
  \left(\frac{\langle\xi_k\rangle}{\langle\xi_{k_2}\rangle}\right)^s
  \langle\xi_k\rangle^{-\delta}
 \ge c_{s,\delta}M^{-2\delta}.
\]
Together with \eqref{eq:4.1.13}, this proves \eqref{eq:2.2.2}.
Since the lower bound tends to infinity as \(M\to\infty\) and both
inputs have unit \(X^s\) norm, no bounded bilinear extension into
\(X^{s-\delta}\) can exist.

For the upper bound, let \(s\ge2\) and \(H_1,H_2\in X^\infty\).
Using \eqref{eq:2.2.1}, the isometry from
Proposition~\ref{prop:2.1.1-unitary-linear-evolution}, and
\eqref{eq:3.17}, we obtain, for \(t>0\),
\begin{align*}
 \|\mathcal B_t(H_1,H_2)\|_{X^{s-1}}
 \le{}&\int_0^t
 \|\mathcal N_2(e^{\tau\mathcal L}H_1,e^{\tau\mathcal L}H_2)\|_{X^{s-1}}
 \,\mathrm d\tau\\*
 &+\int_0^t
 \|\mathcal N_2(e^{\tau\mathcal L}H_2,e^{\tau\mathcal L}H_1)\|_{X^{s-1}}
 \,\mathrm d\tau\\*
 \le{}&C_st\,\|H_1\|_{X^s}\|H_2\|_{X^s}.
\end{align*}
For \(t<0\), the same estimate on \([t,0]\) has \(|t|\) in place
of \(t\), and \(\mathcal B_0=0\). Density of \(X^\infty\) in \(X^s\)
provides the unique bounded bilinear extension satisfying
\eqref{eq:2.2.3}. The argument also applies to complex inputs.
Together with the lower bound, this establishes the sharp
one-derivative loss for \(s\ge2\).
\end{proof}

We next solve the integer phase equations in all four signed channels
and evaluate the corresponding polarized coefficients.

\begin{proof}[Proof of Theorem~\ref{thm:2.2.2-exact-lattice-resonances}]
Fix \(\sigma\in\{\pm1\}\) and \(m_1,m_2\in\mathbb Z_*\) with
\(m:=m_1+m_2\ne0\). By \eqref{eq:3.11}, the resonance condition
reduces in cases \emph{(i)}--\emph{(iv)}, respectively, to
\begin{equation}
 n_1m_2^2+n_2m_1^2=0,
\end{equation}
\begin{equation}
 n_1m_2(m_2+2m_1)-n_2m_1^2=0,
\end{equation}
\begin{equation}
 n_2m_1(m_1+2m_2)-n_1m_2^2=0,
\end{equation}
and
\begin{equation}
 n_1m_2(2m_1+m_2)+n_2m_1(m_1+2m_2)=0.
\end{equation}
Clearing only the nonzero vertical indices retains all horizontal
zero modes. In each of the first three equations, at least one
coefficient is nonzero because \(m_1m_2\ne0\). In the fourth,
simultaneous vanishing would require \(2m_1+m_2=m_1+2m_2=0\),
contrary to \(m_1m_2\ne0\). Hence all four divisors in the theorem
are positive.

For \(A,B\in\mathbb Z\) not both zero, let
\(d:=\gcd(|A|,|B|)\). The integer solutions satisfy
\begin{equation}
\label{eq:4.1.18}
 Ax+By=0
 \quad\Longleftrightarrow\quad
 (x,y)=r\left(\frac{B}{d},-\frac{A}{d}\right),
\end{equation}
where \(r\in\mathbb Z\) is unique. Division by \(d\) leaves coprime
coefficients, so the displayed generator is primitive. If one
coefficient is zero, the other reduced coefficient is \(\pm1\), and
the formula remains valid. Applying \eqref{eq:4.1.18} to the four
phase equations proves
\eqref{eq:2.2.5}, \eqref{eq:2.2.7}, \eqref{eq:2.2.9}, and
\eqref{eq:2.2.11}, with \(r=\ell\) in cases \emph{(i)}, \emph{(iii)},
and \emph{(iv)}, and \(r=-\ell\) in case \emph{(ii)}.
Adding the input indices recovers the stated horizontal output.
Each reduction is reversible, which proves the necessity and
sufficiency of the parametrizations and the uniqueness of \(\ell\).

The primitive generators are nonzero, so a resonance is nontrivial if
and only if \(\ell\ne0\). For such a resonance, neither horizontal
input vanishes in case \emph{(i)}. In case \emph{(ii)}, only the second
input can vanish, precisely when \(m_2=-2m_1\). In case \emph{(iii)},
only the first can vanish, precisely when \(m_1=-2m_2\). In case
\emph{(iv)}, these same conditions characterize \(n_2=0\) and
\(n_1=0\), respectively.
The mixed channels have nonzero horizontal output because
\(n=\frac{\ell m^2}{d_+}\) in case \emph{(ii)} and
\(n=\frac{\ell m^2}{d_-}\) in case \emph{(iii)}.
In cases \emph{(i)} and \emph{(iv)}, the condition \(n=0\) is equivalent
to \(m_1=m_2\), since \(m_1+m_2\ne0\).

We now evaluate \eqref{eq:2.1.38} on these resonances.
In case \emph{(i)}, \eqref{eq:2.2.5} implies
\[
 n_1m_2-n_2m_1
 =\frac{\ell}{d_{\mathrm s}}
   \bigl(m_1^2m_2+m_2^2m_1\bigr)
 =\frac{\ell}{d_{\mathrm s}}m_1m_2m.
\]
The factorization \eqref{eq:3.13} therefore reduces to
\begin{equation}
 \widetilde q_{\sigma\sigma\sigma}
 =\frac{i\kappa_x}{2\sqrt\gamma}
   \frac{\ell m_1m_2m}{d_{\mathrm s}}
   \left(\frac{1}{m_1}-\frac{1}{m_2}\right)
 =\frac{i\kappa_x\ell}{2\sqrt\gamma}
   \frac{m_2^2-m_1^2}{d_{\mathrm s}}
 =-\frac{i\kappa_x n}{2\sqrt\gamma}.
\end{equation}
This proves \eqref{eq:2.2.6}. For \(\ell\ne0\), the coefficient
vanishes precisely when \(m_1^2=m_2^2\). Since \(m_1+m_2\ne0\),
this is equivalent to \(m_1=m_2\).

In case \emph{(ii)}, \((\mu,\nu)=(-\sigma,\sigma)\), and
\eqref{eq:2.1.38} reduces to
\(\widetilde q_{\sigma,-\sigma,\sigma}=\frac{m_{12}}{2\sqrt\gamma}\).
Substituting \eqref{eq:2.2.7}, we find
\[
 n_1m_2-n_2m_1
 =\frac{\ell}{d_+}
   \bigl[m_1^2m_2-m_1m_2(m_2+2m_1)\bigr]
 =-\frac{\ell}{d_+}m_1m_2m.
\]
Consequently,
\begin{equation}
 \widetilde q_{\sigma,-\sigma,\sigma}
 =-\frac{i\kappa_x\ell}{2\sqrt\gamma}\frac{m_2m}{d_+}
 =-\frac{i\kappa_x n}{2\sqrt\gamma}\frac{m_2}{m},
\end{equation}
which is \eqref{eq:2.2.8}. Since \(m_2m\ne0\), this coefficient
vanishes if and only if \(\ell=0\).

Case \emph{(iii)} follows from case \emph{(ii)} by exchanging
\((k_1,\mu)\) and \((k_2,\nu)\). Under this exchange,
\(d_+\) becomes \(d_-\) and \(\ell\) is unchanged. By
\eqref{eq:3.12},
\[
 \widetilde q_{\sigma,\sigma,-\sigma}(k;k_1,k_2)
 =\widetilde q_{\sigma,-\sigma,\sigma}(k;k_2,k_1)
 =-\frac{i\kappa_x\ell}{2\sqrt\gamma}\frac{m_1m}{d_-}
 =-\frac{i\kappa_x n}{2\sqrt\gamma}\frac{m_1}{m}.
\]
This is \eqref{eq:2.2.10}. Since \(m_1m\ne0\), the coefficient
vanishes exactly when \(\ell=0\). Thus every nontrivial mixed-branch
resonance is active, including those with a zero horizontal input.
In case \emph{(iv)}, both indicators in \eqref{eq:2.1.38}
vanish, so \(\widetilde q_{\sigma,-\sigma,-\sigma}=0\) for every
admissible triple, independently of resonance.

Finally, exchanging the input pairs \((k_1,\mu)\) and \((k_2,\nu)\)
in the parametrizations interchanges cases \emph{(ii)} and \emph{(iii)}
with \(\ell\) unchanged, and maps cases \emph{(i)} and \emph{(iv)}
to themselves with \(\ell\) replaced by \(-\ell\).
Under the simultaneous reflection \(k_j\mapsto-k_j\), \(j=1,2\), and
\((\sigma,\mu,\nu)\mapsto(-\sigma,-\mu,-\nu)\), the divisors and
primitive generators are unchanged. The horizontal indices change sign,
so \(\ell\) is replaced by \(-\ell\) in all four cases.
The corresponding coefficient identities follow from
\eqref{eq:3.12} and \eqref{eq:3.16}.
\end{proof}

The support classification requires cancellation of the averaged
corrections as well as the local transport terms. We verify both on
lattice rays and then prove that transparency forces collinearity.

\begin{proof}[Proof of Theorem~\ref{thm:2.2.3-maximal-transparent-rays}]
Let \(k_0=(n_0,m_0)\in K\) be primitive and set
\(\lambda_0:=\frac{L_zn_0}{L_xm_0}\). For smooth scalar functions
\(f,g\) with Fourier support in \(\Gamma_{k_0}\), the multiplier
identities are
\begin{equation}
\label{eq:4.1.21}
 \partial_z^{-1}\partial_xf=\lambda_0 f,
\end{equation}
and \(\partial_xg=\lambda_0\partial_zg\). Therefore,
\begin{equation}
\label{eq:4.1.22}
        b(f,g)
        =(\partial_z^{-1}\partial_xf)\partial_zg-f\partial_xg=0.
\end{equation}

A pair \(j_1k_0,j_2k_0\) contributes to \(\overline{fg}\) only if
\((j_1+j_2)m_0=0\). Then \(j_1+j_2=0\), so the horizontal output
also vanishes. Hence
\begin{equation}
\label{eq:4.1.23}
        \partial_x\overline{fg}=0.
\end{equation}
Equations \eqref{eq:4.1.22}--\eqref{eq:4.1.23} and
\eqref{eq:2.1.5} imply \(\mathcal P b(f,g)=0\). For smooth states
\(U_j\) supported on the ray, we therefore have
\(\mathcal N_2(U_1,U_2)=0\). The factor
\(\partial_x\overline{v_1\theta_2}\) in \eqref{eq:2.1.8} also
vanishes, so \(\mathcal N_3(U_1,U_2,U_3)=0\).
This proves the ordered identities \eqref{eq:2.2.14} and the implication
\emph{(iii)} \(\Rightarrow\) \emph{(i)}.

For the converse, assume \emph{(i)} and suppose that \(S\) contains
two frequencies \(k_j=(n_j,m_j)\), \(j=1,2\), such that
\[
        n_1m_2-n_2m_1\ne0.
\]
Since at least one of \(m_1+m_2\) and \(m_1-m_2\) is nonzero,
choose \(\varepsilon\in\{\pm1\}\) such that
\(q:=k_1+\varepsilon k_2\in K\). The symmetry \(S=-S\) ensures
that both input frequencies lie in \(S\). Define
\begin{equation}
 v=e^{i\xi_{k_1}\cdot(x,z)}+e^{i\xi_{-k_1}\cdot(x,z)},
\end{equation}
and
\begin{equation}
 \theta=e^{i\xi_{\varepsilon k_2}\cdot(x,z)}+e^{i\xi_{-\varepsilon k_2}\cdot(x,z)}.
\end{equation}
Then \(U=(v,\theta)\) is a real state in \(\mathcal R_S^\infty\).
Only the ordered pair
\((k_1,\varepsilon k_2)\) contributes to the Fourier coefficient of
\(b(v,\theta)\) at \(q\). Indeed, noncollinearity implies that
\(\eta_1k_1+\eta_2\varepsilon k_2=k_1+\varepsilon k_2\), with
\(\eta_j\in\{\pm1\}\), holds only when \(\eta_1=\eta_2=1\).
The vertical component of \(q\) is nonzero, so \(\mathcal P\) leaves
this coefficient unchanged. Consequently, \eqref{eq:3.6} implies
\begin{equation}
\label{eq:4.1.26}
 \widehat{[\mathcal N_2(U,U)]_\theta}(q)
 =i\varepsilon\kappa_x
   \frac{n_1m_2-n_2m_1}{m_1}\ne0.
\end{equation}
This contradicts \eqref{eq:2.2.13} and proves
\emph{(i)} \(\Rightarrow\) \emph{(ii)}.

For \emph{(ii)} \(\Rightarrow\) \emph{(iii)}, divide a fixed
\(k_*\in S\) by the greatest common divisor of its components to
obtain a primitive \(k_0=(n_0,m_0)\), with \(m_0\ne0\).
For any \((n,m)\in S\), collinearity gives \(nm_0=mn_0\).
Coprimality implies \(m_0\mid m\), hence \((n,m)=j(n_0,m_0)\)
for some \(j\in\mathbb Z_*\). Therefore \(S\subset\Gamma_{k_0}\).

For every \(s\in\mathbb R\), inclusion of symmetric support sets is
equivalent to inclusion of their spaces \(\mathcal R_S^s\), since each
allowed frequency is realized by a real trigonometric polynomial.
Let \(S\) be a nonempty transparent set maximal under inclusion.
By \emph{(iii)}, \(S\subset\Gamma_{k_0}\), and the first
part of the proof shows that \(\Gamma_{k_0}\) is transparent. Hence
\(S=\Gamma_{k_0}\). Conversely, any proper symmetric enlargement of
\(\Gamma_{k_0}\) contains a frequency not collinear with \(k_0\).
The nonzero output in \eqref{eq:4.1.26} rules out transparency of
that enlargement. Finally, two primitive vectors generate the same
ray exactly when they differ by a sign.
\end{proof}

We conclude the main proofs with
Theorem~\ref{thm:2.2.4-mean-resonant-averaging}.
We first use the cancellations at zero horizontal output, then identify
the resonant average and bound the nonresonant remainder.

\begin{proof}[Proof of Theorem~\ref{thm:2.2.4-mean-resonant-averaging}]
We first prove \emph{(i)}. Let \(k=(0,m)\in K\).
By Theorem~\ref{thm:2.2.2-exact-lattice-resonances}, the same-branch
resonant coefficient is proportional to \(n\), and every nontrivial
mixed-branch resonance has \(n\ne0\). Trivial resonances have
\(n_1=n_2=0\), hence zero coefficient. The opposite-input channel
is identically transparent. This proves \eqref{eq:2.2.16}.

For a nonresonant interaction at this output, substituting
\(n_2=-n_1\) into \eqref{eq:2.1.36}--\eqref{eq:2.1.38} gives
\begin{equation}
\label{eq:4.3.1}
\begin{aligned}
 \Phi_{\sigma\mu\nu}
 ={}&-\frac{N\kappa_xn_1}{\kappa_z}
       \left(\frac\mu{m_1}-\frac\nu{m_2}\right),\\
 \widetilde q_{\sigma\mu\nu}
 ={}&\frac{i\kappa_xn_1m}{2\sqrt\gamma}
       \left(\frac{\mathbf1_{\{\nu=\sigma\}}}{m_1}
             -\frac{\mathbf1_{\{\mu=\sigma\}}}{m_2}\right).
\end{aligned}
\end{equation}
Division in each of the four channels leads to
\begin{equation}
\label{eq:4.3.2}
 \frac{\widetilde q_{\sigma\mu\nu}}{i\Phi_{\sigma\mu\nu}}
 =\frac{\kappa_z}{2\sigma N\sqrt\gamma}
 \begin{cases}
  -m,&(\mu,\nu)=(\sigma,\sigma),\\
  m_2,&(\mu,\nu)=(-\sigma,\sigma),\\
  m_1,&(\mu,\nu)=(\sigma,-\sigma),\\
  0,&(\mu,\nu)=(-\sigma,-\sigma).
 \end{cases}
\end{equation}
Every canceled factor is nonzero on the set
\(\Phi_{\sigma\mu\nu}\ne0\). In particular, \(n_1\ne0\), and
\(m_1\ne m_2\) in the same-branch channel. Formula
\eqref{eq:4.3.2} is used only on this nonresonant set.

At \(n=0\), the external phase in \eqref{eq:3.15} equals one.
The resonant contributions vanish by \eqref{eq:2.2.16}, while each
nonresonant contribution has multiplier
\(
 (1-e^{-it\Phi_{\sigma\mu\nu}})
       \frac{\widetilde q_{\sigma\mu\nu}}{i\Phi_{\sigma\mu\nu}}.
\)
By \eqref{eq:4.3.2}, its absolute value is at most
\(C(|m_1|+|m_2|)\), uniformly for \(t\in\mathbb R\).
Set this multiplier to zero on all other output modes and on
resonant pairs. Lemma~\ref{lem:3.3-fourier-summation}, with \(d=1\),
and the wave-energy identity \eqref{eq:2.1.29} prove
\eqref{eq:2.2.17} for trigonometric polynomials and every \(s\ge2\).
For each fixed \(t\), density identifies this extension with the
horizontal mean of \(\mathcal B_t\) from
Theorem~\ref{thm:2.2.1-second-picard-obstruction}. The bound is
independent of \(t\), which proves \eqref{eq:2.2.17} for all
inputs in \(X^s\), uniformly in time.

To prove \emph{(ii)}, we first take constrained trigonometric
polynomial inputs. For \(t\ne0\), the group property implies
\begin{equation}
\label{eq:4.3.3}
 \frac{1}{t}e^{-t\mathcal L}\mathcal B_t(H_1,H_2)
 =\frac{1}{t}\int_0^t e^{-\tau\mathcal L}
 \Bigl[\mathcal N_2(e^{\tau\mathcal L}H_1,e^{\tau\mathcal L}H_2)
       +\mathcal N_2(e^{\tau\mathcal L}H_2,e^{\tau\mathcal L}H_1)\Bigr]
       \,\mathrm d\tau.
\end{equation}
Thus the normalized response is the time average of the quadratic
interaction conjugated by the linear group. For the chosen inputs,
\eqref{eq:3.15} is a finite sum. After multiplication by
\(\frac{1}{t}e^{-t\mathcal L}\), each resonant multiplier equals
\(\widetilde q_{\sigma\mu\nu}\), whereas each nonresonant multiplier
satisfies
\[
 \frac{1-e^{-it\Phi_{\sigma\mu\nu}}}{it\Phi_{\sigma\mu\nu}}
       \widetilde q_{\sigma\mu\nu}
 \longrightarrow0\qquad\text{as }|t|\to\infty.
\]
Thus the normalized response converges to
\(\mathcal Q_{\mathrm{res}}(H_1,H_2)\) in every Sobolev space for
these fixed polynomial inputs. The bound
\eqref{eq:2.2.3} and linear unitarity imply, for \(s\ge2\),
\[
 \sup_{t\ne0}
 \left\|\frac{1}{t} e^{-t\mathcal L}\mathcal B_t(H_1,H_2)\right\|_{X^{s-1}}
 \le C_s\|H_1\|_{X^s}\|H_2\|_{X^s}.
\]
Passing to the limit on the finite Fourier support proves
\eqref{eq:2.2.18} for polynomial inputs. Density then extends
\(\mathcal Q_{\mathrm{res}}\) uniquely to a bounded bilinear map from
\(X^s\times X^s\) to \(X^{s-1}\). For smooth inputs, rapid Fourier
decay and the polynomial growth of \(\widetilde q\) identify the
extension with \eqref{eq:2.2.15}. The map is symmetric by
\eqref{eq:3.12}. It preserves real-valued inputs by
\eqref{eq:3.16} and the invariance of the resonant set under
simultaneous frequency and branch reflection. Its horizontal mean
vanishes by \eqref{eq:2.2.16}.

For general \(H_1,H_2\in X^s\), choose constrained trigonometric
polynomials \(H_j^{(J)}\to H_j\) in \(X^s\), \(j=1,2\).
Bilinearity, the preceding uniform bound, and \eqref{eq:2.2.18}
bound the sum of the approximation errors in the normalized response
and the resonant operator by
\[
 C_s\Bigl(\|H_1-H_1^{(J)}\|_{X^s}\|H_2\|_{X^s}
       +\|H_1^{(J)}\|_{X^s}\|H_2-H_2^{(J)}\|_{X^s}\Bigr),
\]
independently of \(t\ne0\). For each fixed \(J\), the polynomial
responses converge as \(|t|\to\infty\). Letting first \(|t|\to\infty\)
and then \(J\to\infty\) proves \eqref{eq:2.2.19} in \(X^{s-1}\).

For \emph{(iii)}, return to polynomial inputs \(H_1,H_2\) and set
\(Z_{j,\mu}=Z_\mu[H_j]\). Omitting the input arguments in
\(\mathcal B_t\) and \(\mathcal Q_{\mathrm{res}}\) in the next
display, we subtract the resonant contribution from
\eqref{eq:3.15} to obtain, for \(k\in K\),
\begin{equation}
\label{eq:4.3.4}
 \widehat{Z_\sigma[e^{-t\mathcal L}\mathcal B_t
                         -t\mathcal Q_{\mathrm{res}}]}(k)
 =\sum_{\mu,\nu=\pm}
     \sum_{\substack{k_1+k_2=k\\\Phi_{\sigma\mu\nu}\ne0}}
     \frac{1-e^{-it\Phi_{\sigma\mu\nu}}}{i\Phi_{\sigma\mu\nu}}\times\widetilde q_{\sigma\mu\nu}
     \widehat Z_{1,\mu}(k_1)\widehat Z_{2,\nu}(k_2).
\end{equation}
By \(|1-e^{-it\Phi}|\le2\) and
Lemma~\ref{lem:3.2-phase-division}, each multiplier on the
right-hand side is bounded in absolute value by
\(
 C\bigl(\langle\xi_{k_1}\rangle
        +\langle\xi_{k_2}\rangle\bigr)^4,
\)
uniformly in \(t\). Applying
Lemma~\ref{lem:3.3-fourier-summation} with \(d=4\) to each branch
and summing the resulting bounds, we obtain \eqref{eq:2.2.20}
for \(s\ge4\) from the wave-energy identity \eqref{eq:2.1.29}.
At each fixed time, the extensions of \(\mathcal B_t\) and
\(\mathcal Q_{\mathrm{res}}\) identify the extended remainder
with the left-hand side of \eqref{eq:4.3.4}. Density thus proves
\eqref{eq:2.2.20} for arbitrary \(H_1,H_2\in X^s\), with a
constant independent of time. Division by \(|t|\) establishes
\eqref{eq:2.2.21} for every \(t\ne0\). The same arguments apply
to complexified inputs.
\end{proof}

\subsection{Anisotropic and dyadic bounds}
\label{subsec:4.2-anisotropic-dyadic}

We now refine the fixed-time bounds of
Theorem~\ref{thm:2.2.1-second-picard-obstruction} in two directions.
A mixed-branch family with fixed vertical indices isolates diagnostic
transport and excludes compensation by finite vertical Sobolev
regularity. We then use the same-branch family \eqref{eq:4.1.1}
to obtain matching bounds on physical dyadic shells.

\begin{corollary}
\label{cor:4.2.1-diagnostic-mixed-resonances}
For every integer \(q\ge1\), set
\begin{equation}
\label{eq:4.2.1}
 (\ell_0,h_q,o_q):=\bigl((0,-2),\,(q,1),\,(q,-1)\bigr).
\end{equation}
Then \(o_q=\ell_0+h_q\), and for each \(\sigma\in\{\pm1\}\) the
triple is an active exact resonance in the channel
\((\sigma,-\sigma)\to\sigma\), with
\begin{equation}
\label{eq:4.2.2}
 \Phi_{\sigma,\sigma,-\sigma}(o_q;\ell_0,h_q)=0,
\end{equation}
and
\begin{equation}
 \widetilde q_{\sigma,\sigma,-\sigma}(o_q;\ell_0,h_q)=-\frac{i\kappa_xq}{\sqrt\gamma}.
\end{equation}
The physical frequencies satisfy
\[
 |\xi_{o_q}|=|\xi_{h_q}|
 \quad\text{and}\quad
 \langle\xi_{o_q}\rangle=\langle\xi_{h_q}\rangle\simeq q,
\]
with comparison constants depending only on \(L_x,L_z\).
At the output \((o_q,\sigma)\), only diagnostic transport contributes
to the polarized coefficient. The horizontal transport and quadratic
averaging contributions vanish.
\end{corollary}
\begin{proof}
The vertical indices \(-2,1,-1\) are nonzero. Substitution into
\eqref{eq:2.1.36} gives
\[
 \frac{\Phi_{\sigma,\sigma,-\sigma}(o_q;\ell_0,h_q)}{\sigma N}
 =\lambda(o_q)-\lambda(\ell_0)+\lambda(h_q)
 =\frac{L_z}{L_x}(-q-0+q)=0.
\]
For this branch combination, \eqref{eq:2.1.38} selects the reversed
input order:
\[
 \widetilde q_{\sigma,\sigma,-\sigma}(o_q;\ell_0,h_q)
 =\frac{m_b(h_q,\ell_0)}{2\sqrt\gamma}
 =\frac{i\kappa_x}{2\sqrt\gamma}
   \left(\frac{q(-2)}{1}-0\right)
 =-\frac{i\kappa_xq}{\sqrt\gamma}.
\]
Because \(\ell_0\) has horizontal index zero, the coefficient
\(m_b(h_q,\ell_0)=\lambda(h_q)i\kappa_z(-2)\) arises solely from
diagnostic transport. The quadratic averaging correction is supported
at zero vertical output and therefore does not contribute here.
Finally,
\[
 |\xi_{o_q}|^2=|\xi_{h_q}|^2=\kappa_x^2q^2+\kappa_z^2,
\]
which proves the frequency comparisons.
\end{proof}

Fix \(s\in\mathbb R\) and \(\sigma\in\{\pm1\}\). For the
frequencies in \eqref{eq:4.2.1}, choose
\[
 A_0:=(L_xL_z)^{-\frac12}\langle\xi_{\ell_0}\rangle^{-s}
 \quad\text{and}\quad
 A_q:=(L_xL_z)^{-\frac12}\langle\xi_{h_q}\rangle^{-s},
\]
and define
\begin{equation}
\label{eq:4.2.4}
 H_0:=E_\sigma(\ell_0;A_0),
\end{equation}
and
\begin{equation}
\label{eq:4.2.5}
 G_q:=E_{-\sigma}(h_q;A_q).
\end{equation}
By \eqref{eq:4.1.6}, both inputs have unit \(X^s\) norm, and
\(H_0\) is independent of \(q\). Their branch assignments follow
from \eqref{eq:4.1.5}: the frequencies \(\ell_0\) and \(h_q\)
occupy the branches \(\sigma\) and \(-\sigma\), respectively,
with the conjugate branches at the negative frequencies.

The possible spatial output frequencies are
\[
 \bigl\{\pm(\ell_0+h_q),\,\pm(\ell_0-h_q)\bigr\}
 =\bigl\{\pm o_q,\,\pm(q,3)\bigr\}.
\]
These frequencies are pairwise distinct for \(q\ge1\). Hence only
\((\ell_0,h_q)\), with branches \((\sigma,-\sigma)\), contributes
to \((o_q,\sigma)\) in \eqref{eq:3.15}. The two output branches
contribute separately to the norm in \eqref{eq:2.1.29}, so the
other branch cannot cancel this coefficient.

By \eqref{eq:4.2.2}, the time integral in \eqref{eq:3.15} equals \(t\)
for every \(t\in\mathbb R\). Thus
\begin{equation}
\label{eq:4.2.6}
 \widehat{Z_\sigma[\mathcal B_t(H_0,G_q)]}(o_q)
 =-\frac{i\kappa_xq}{\sqrt\gamma}\,
   t e^{it\omega_\sigma(o_q)}A_0A_q.
\end{equation}
Retaining this coefficient and its conjugate at \((-o_q,-\sigma)\)
in \eqref{eq:2.1.29}, we obtain, for \(0\le\delta<1\),
\[
 \|\mathcal B_t(H_0,G_q)\|_{X^{s-\delta}}
 \ge\frac{\kappa_x|t|q}
 {\sqrt{\gamma L_xL_z}\,\langle\xi_{\ell_0}\rangle^s
                            \langle\xi_{h_q}\rangle^\delta}
 \ge c_{s,\delta}|t|q^{1-\delta}.
\]
Here \(\langle\xi_{o_q}\rangle=\langle\xi_{h_q}\rangle\), so the
weights of order \(s\) cancel for every real \(s\). The constant is
independent of \(q,t,\sigma\).

Since the vertical indices \(-2,1,-1\) are fixed, their contributions
to the weights in \eqref{eq:2.1.20} are independent of \(q\).
Renormalizing the same inputs in the tensor-product spaces therefore
establishes the following anisotropic obstruction.

\begin{corollary}
\label{cor:4.2.2-anisotropic-picard-obstruction}
Let \(s_x,s_z,r_z\in\mathbb R\), \(0\le\delta_x<1\), and
\(t\in\mathbb R\setminus\{0\}\). The operator \(\mathcal B_t\), initially
defined on \(X^\infty\times X^\infty\), has no bounded real-bilinear
extension
\begin{equation}
\label{eq:4.2.7}
 \mathcal B_t:X^{s_x,s_z}\times X^{s_x,s_z}
       \longrightarrow X^{s_x-\delta_x,r_z}.
\end{equation}
Thus no finite choice of the vertical Sobolev indices permits a loss of
less than one horizontal derivative in these spaces.
\end{corollary}

\begin{proof}
Take the frequencies \eqref{eq:4.2.1} and the inputs
\eqref{eq:4.2.4}--\eqref{eq:4.2.5}, with amplitudes
\[
 A_0:=(L_xL_z)^{-\frac12}\langle2\kappa_z\rangle^{-s_z}
 \quad\text{and}\quad
 A_q:=(L_xL_z)^{-\frac12}
       \langle\kappa_xq\rangle^{-s_x}\langle\kappa_z\rangle^{-s_z}.
\]
By \eqref{eq:4.1.5} and the tensor-product wave-energy
identity \eqref{eq:2.1.32},
\[
 \|H_0\|_{X^{s_x,s_z}}=\|G_q\|_{X^{s_x,s_z}}=1.
\]
The supports and branch assignments remain unchanged, so
\eqref{eq:4.2.6} holds with these amplitudes. Retaining the selected
coefficient and its conjugate in the tensor-product norm, we obtain
\[
 \|\mathcal B_t(H_0,G_q)\|_{X^{s_x-\delta_x,r_z}}
 \ge\sqrt{L_xL_z}\langle\kappa_xq\rangle^{s_x-\delta_x}
                 \langle\kappa_z\rangle^{r_z}
 \left|\widehat{Z_\sigma[\mathcal B_t(H_0,G_q)]}(o_q)\right|.
\]
Inserting \eqref{eq:4.2.6}, we find
\begin{equation}
\label{eq:4.2.8}
 \|\mathcal B_t(H_0,G_q)\|_{X^{s_x-\delta_x,r_z}}
 \ge\frac{\kappa_x|t|}{\sqrt{\gamma L_xL_z}}
       \frac{\langle\kappa_z\rangle^{r_z-s_z}}
            {\langle2\kappa_z\rangle^{s_z}}
       q\langle\kappa_xq\rangle^{-\delta_x}.
\end{equation}
The horizontal weights of order \(s_x\) cancel, while the vertical
factor is a fixed positive constant for every \(s_z,r_z\in\mathbb R\).
Since \(q\langle\kappa_xq\rangle^{-\delta_x}\simeq q^{1-\delta_x}
\to\infty\), the unit-norm smooth inputs contradict any bounded
extension in \eqref{eq:4.2.7}.
\end{proof}

For the dyadic bounds, we return to the same-branch family
\eqref{eq:4.1.1}. Superposing its high-frequency inputs within a single
physical annulus will produce a matching lower bound on every
sufficiently high dyadic shell.

Choose \(\varphi\in C_c^\infty((\frac{1}{2},2))\) with
\(0\le\varphi\le1\) and \(\varphi\equiv1\) on
\([\frac78,\frac98]\). Let \(\chi:[0,\infty)\to[0,1]\) be smooth,
supported in \([0,2)\), and equal to one on \([0,1]\).
For \(j\in\mathbb Z\), define the Fourier localizations
\begin{equation}
\label{eq:4.2.9}
 \widehat{\mathcal P_jf}(k):=\varphi\bigl(2^{-j}|\xi_k|\bigr)\widehat f(k),
\end{equation}
and
\begin{equation}
\label{eq:4.2.10}
 \widehat{\mathcal P_{\le j}f}(k):=\chi\bigl(2^{-j}|\xi_k|\bigr)\widehat f(k),
\end{equation}
and apply them componentwise. These operators preserve real-valuedness
and the mean constraints, and commute with both the wave transformation
and the linear group. For each \(s\in\mathbb R\), take the normalized
low-frequency input \(H\) from \eqref{eq:4.1.9}, with amplitude
\eqref{eq:4.1.7}. Fix an integer \(j_0\) such that
\begin{equation}
\label{eq:4.2.11}
        \mathcal P_{\le j_0}H=H,
\end{equation}
which holds whenever \(2^{j_0}\ge|\xi_{k_1}|\). Thus \(j_0\)
can be fixed independently of \(s\).

\begin{proposition}
\label{prop:4.2.1-dyadic-picard-bounds}
Let \(s\in\mathbb R\) and \(0\le\delta\le1\), and take \(H\)
and \(j_0\) as above. There exist \(j_*\in\mathbb N\) and real-valued
trigonometric polynomials \(G_j\in X^\infty\), \(j\ge j_*\), such that
\begin{equation}\label{eq:4.2.12}
 \mathcal P_jG_j=G_j,
\end{equation}
and
\begin{equation}\label{eq:4.2.13}
 \|H\|_{X^s}=\|G_j\|_{X^s}=1,
\end{equation}
and, for every \(t\in\mathbb R\setminus\{0\}\),
\begin{equation}
\label{eq:4.2.14}
 \bigl\|\mathcal P_j\mathcal B_t
       \bigl(\mathcal P_{\le j_0}H,\mathcal P_jG_j\bigr)\bigr\|_{X^{s-\delta}}
 \ge c|t|\,2^{j(1-\delta)}.
\end{equation}
Here \(c>0\) is independent of \(j\) and \(t\). Each \(G_j\) can be
chosen as a normalized sum of \(J_j\simeq2^{\frac{j}{2}}\)
high-frequency inputs forming exact same-branch resonances with the
common low input \(H\). Their frequencies and the target output
frequencies are pairwise distinct and lie in the same physical annulus.

If in addition \(s\ge2\), consider the localized bilinear map
\begin{equation}
 (H_1,H_2)\longmapsto\mathcal P_j\mathcal B_t
   \bigl(\mathcal P_{\le j_0}H_1,\mathcal P_jH_2\bigr).
\end{equation}
For all \(j\ge j_*\) and \(t\ne0\), its norm satisfies
\begin{equation}
\label{eq:4.2.16}
 c|t|\,2^{j(1-\delta)}
 \le
 \bigl\|\mathcal P_j\mathcal B_t(\mathcal P_{\le j_0}\,\cdot,\mathcal P_j\,\cdot)\bigr\|_{X^s\times X^s\to X^{s-\delta}}
 \le
 C_s|t|\,2^{j(1-\delta)}.
\end{equation}
The constants may depend on \(s,\delta,L_x,L_z,g,\beta,\gamma\) and
the fixed cutoffs in \eqref{eq:4.2.9}--\eqref{eq:4.2.10}, but not
on \(j\) or \(t\). This includes the endpoint \(\delta=1\).
\end{proposition}

\begin{proof}
Write \(k_2(M)\) and \(k(M)\) for the high input and output
frequencies in \eqref{eq:4.1.1}. For each \(j\), set
\begin{equation}
\label{eq:4.2.17}
 \mathcal M_j:=\left\{M\in\mathbb N:M\ge2,\quad
 \frac{15}{16}2^j\le\kappa_xM^2\le\frac{17}{16}2^j\right\}.
\end{equation}
For large \(j\), the defining interval has endpoints
\(\sqrt{\frac{15}{16\kappa_x}}\,2^{\frac j2}\) and
\(\sqrt{\frac{17}{16\kappa_x}}\,2^{\frac j2}\). Its length is
comparable to \(2^{\frac j2}\), so
\(J_j:=|\mathcal M_j|\simeq2^{\frac j2}\).

The physical frequencies satisfy
\begin{equation}
\label{eq:4.2.18}
 0\le |\xi_{k_2(M)}|-\kappa_xM^2
 \le\frac{\kappa_z^2}{2\kappa_x},
\end{equation}
and
\begin{equation}
\label{eq:4.2.19}
 \bigl||\xi_{k(M)}|-\kappa_xM^2\bigr|
 \le \bigl|\xi_{k(M)}-(-\kappa_xM^2,0)\bigr|
 \le C(1+M).
\end{equation}
On \(\mathcal M_j\), we have \(M\simeq2^{\frac j2}\), so the
errors in \eqref{eq:4.2.18}--\eqref{eq:4.2.19} are \(o(2^j)\).
The fixed margins in \eqref{eq:4.2.17} therefore allow us to choose
\(j_*\) such that
\begin{equation}
\label{eq:4.2.20}
 \frac{7}{8}\,2^j
 \le |\xi_{k_2(M)}|,|\xi_{k(M)}|
 \le\frac{9}{8}\,2^j,
 \qquad M\in\mathcal M_j,
 \qquad j\ge j_*.
\end{equation}
Thus the symbol of \(\mathcal P_j\) equals one at all high inputs and
target outputs.

For each \(M\in\mathcal M_j\), take \(G_M\) from
\eqref{eq:4.1.10}, with amplitude \(A_M\) as in
\eqref{eq:4.1.8}, and define
\begin{equation}
        G_j:=J_j^{-\frac{1}{2}}\sum_{M\in\mathcal M_j}G_M.
\end{equation}
The summands have disjoint Fourier supports. Orthogonality in \(X^s\)
and \eqref{eq:4.1.11} give
\[
        \|G_j\|_{X^s}^2
        =J_j^{-1}\sum_{M\in\mathcal M_j}\|G_M\|_{X^s}^2
        =1.
\]
By \eqref{eq:4.2.20}, we also have \(\mathcal P_jG_j=G_j\).
Together with \eqref{eq:4.1.11}, this proves
\eqref{eq:4.2.12}--\eqref{eq:4.2.13}.

To separate the target coefficients, we check that no other output
coincides with \(k(M)\). Besides \(\pm k(M)\), the possible outputs are
\[
 h(M):=k_1-k_2(M)=(M^2+1,1-M)
 \quad\text{and}\quad -h(M).
\]
The frequencies \(k(M)\) are pairwise distinct and have positive
vertical components, so none equals \(-k(M')\). Their vertical
components also exclude \(k(M)=h(M')\). If \(k(M)=-h(M')\),
then the vertical components give \(M'=M+2\), whereas the horizontal
components require \((M+2)^2-M^2=-2\), a contradiction.
Reflection excludes the remaining collisions. Thus only
\((k_1,k_2(M))\) contributes to \((k(M),+)\), with both input
orders included by polarization. The coefficient at \((-k(M),-)\)
is its complex conjugate.

The target modes \((k(M),+)\) and \((-k(M),-)\) are mutually
orthogonal and fixed by \(\mathcal P_j\). Linearity in the second
input shows that the absolute value of each coefficient is
\(J_j^{-\frac12}\) times the value in \eqref{eq:4.1.12}.
The wave-energy identity \eqref{eq:2.1.29} therefore implies
\begin{equation}
 \bigl\|\mathcal P_j\mathcal B_t(H,G_j)\bigr\|_{X^{s-\delta}}^2
 \ge\frac{\kappa_x^2t^2}{4\gamma L_xL_z}\,
   \frac{1}{J_j}
   \sum_{M\in\mathcal M_j}
   (M^2-1)^2
   \frac{\langle\xi_{k(M)}\rangle^{2(s-\delta)}}
   {\langle\xi_{k_1}\rangle^{2s}
    \langle\xi_{k_2(M)}\rangle^{2s}}.
\label{eq:4.2.22}
\end{equation}
By \eqref{eq:4.2.17} and \eqref{eq:4.2.20},
\(M^2\simeq\langle\xi_{k_2(M)}\rangle
\simeq\langle\xi_{k(M)}\rangle\simeq2^j\), uniformly on
\(\mathcal M_j\). These two-sided comparisons remain valid for every real Sobolev
index. Each summand in \eqref{eq:4.2.22} is bounded below by
\(c\,2^{2j(1-\delta)}\), so summation over \(\mathcal M_j\) proves
\[
 \bigl\|\mathcal P_j\mathcal B_t(H,G_j)\bigr\|_{X^{s-\delta}}
 \ge c|t|2^{j(1-\delta)}.
\]
Together with \eqref{eq:4.2.11} and \(\mathcal P_jG_j=G_j\), this proves
\eqref{eq:4.2.14} and the lower bound in \eqref{eq:4.2.16}.

For the upper bound, assume \(s\ge2\). For \(j\ge j_*\), the symbol
of \(\mathcal P_j\) is supported where
\(\langle\xi_k\rangle^{1-\delta}\le C2^{j(1-\delta)}\). Thus, for
\(F\in X^{s-1}\),
\begin{equation}
\label{eq:4.2.23}
        \|\mathcal P_jF\|_{X^{s-\delta}}
        \le C2^{j(1-\delta)}\|F\|_{X^{s-1}}.
\end{equation}
Both cutoffs have operator norm at most one on every \(X^r\).
Combining \eqref{eq:4.2.23} with the one-derivative estimate
\eqref{eq:2.2.3}, we conclude that
\begin{align*}
 \bigl\|\mathcal P_j\mathcal B_t(\mathcal P_{\le j_0}H_1,\mathcal P_jH_2)\bigr\|_{X^{s-\delta}}
 \le{}& C2^{j(1-\delta)}
       \|\mathcal B_t(\mathcal P_{\le j_0}H_1,\mathcal P_jH_2)\|_{X^{s-1}}\\*
 \le{}& C_s|t|2^{j(1-\delta)}
       \|H_1\|_{X^s}\|H_2\|_{X^s}.
\end{align*}
This proves the upper bound in \eqref{eq:4.2.16}.
\end{proof}



\section{Consequences for the evolution problem}
\label{sec:5-iteration-obstructions-fourier-rays}

We derive consequences of the Picard obstruction for quadratic
iteration, trajectory-map differentiability, and quadratic changes of
variables. On transparent Fourier rays, the nonlinear terms vanish.
We use this cancellation to construct explicit global solutions and
characterize uniform strict stratification.

\subsection{Iteration and trajectory maps}

The first obstruction concerns trajectory-space estimates and does not
require a nonlinear solution map. For smooth trajectories \(u,v\), define
\begin{equation}
\label{eq:5.1.1}
 \mathfrak B(u,v)(t)
 :=\int_0^te^{(t-\tau)\mathcal L}
 \Bigl[
 \mathcal N_2\bigl(u(\tau),v(\tau)\bigr)
 +\mathcal N_2\bigl(v(\tau),u(\tau)\bigr)
 \Bigr] \,\mathrm d\tau.
\end{equation}
On free trajectories, this operator agrees with the second Picard
operator:
\[
\mathfrak B(e^{\cdot\mathcal L}H_1,e^{\cdot\mathcal L}H_2)(t) =\mathcal B_t(H_1,H_2).
\]
The fixed-time obstruction in
Theorem~\ref{thm:2.2.1-second-picard-obstruction} therefore excludes the
following combination of trajectory-space bounds.

\begin{theorem}
\label{thm:5.1.1-quadratic-iteration-obstruction}
Let \(T>0\), \(s\in\mathbb R\), and \(0\le\delta<1\). No real Banach
space \(Y_T\) of trajectories satisfies the following conditions with
constants \(C_L,C_B>0\):
\begin{enumerate}[(i)]
\item \(Y_T\) is continuously embedded in
      \(C([0,T];X^{s-\delta})\);
\item for every \(H\in X^\infty\), the free trajectory
      \(e^{\cdot\mathcal L}H\) belongs to \(Y_T\), and
      \begin{equation}
      \label{eq:5.1.2}
      \|e^{\cdot\mathcal L}H\|_{Y_T}
      \le C_L\|H\|_{X^s};
      \end{equation}
\item the map \(\mathfrak B\), defined by \eqref{eq:5.1.1} on pairs of
      free trajectories with initial data in \(X^\infty\), extends to a
      bounded real-bilinear map from \(Y_T\times Y_T\) to \(Y_T\), with
      \begin{equation}
      \label{eq:5.1.3}
      \|\mathfrak B(u,v)\|_{Y_T}
      \le C_B\|u\|_{Y_T}\|v\|_{Y_T}.
      \end{equation}
\end{enumerate}
In particular, neither an ordered bilinear Duhamel estimate nor a
quadratic estimate for the self-interaction Duhamel map can close the
usual Picard scheme in such a space.
\end{theorem}

\begin{proof}
Assume that such a space exists. Denote by \(C_E\) the embedding
constant in \emph{(i)}, and fix \(t_0\in(0,T]\). For
\(H_1,H_2\in X^\infty\),
\eqref{eq:5.1.2} and \eqref{eq:5.1.3} imply
\begin{align*}
 \|\mathcal B_{t_0}(H_1,H_2)\|_{X^{s-\delta}}
 \le{}&C_E\|\mathfrak B(e^{\cdot\mathcal L}H_1,
                              e^{\cdot\mathcal L}H_2)\|_{Y_T}\\*
 \le{}&C_EC_BC_L^2\|H_1\|_{X^s}\|H_2\|_{X^s}.
\end{align*}
By density, this estimate extends \(\mathcal B_{t_0}\) to a bounded
bilinear map of the type excluded by
Theorem~\ref{thm:2.2.1-second-picard-obstruction}.

Summing the two input orders converts an ordered bilinear Duhamel
bound into the polarized estimate. For the self-interaction map,
\[
 \frac12\mathfrak B(u,u)(t)
 =\int_0^te^{(t-\tau)\mathcal L}
       \mathcal N_2\bigl(u(\tau),u(\tau)\bigr)\,\mathrm d\tau.
\]
Suppose that
\(\frac12\|\mathfrak B(u,u)\|_{Y_T}\le C\|u\|_{Y_T}^2\).
For nonzero smooth free trajectories \(u,v\), set
\(a:=\bigl(\frac{\|v\|_{Y_T}}{\|u\|_{Y_T}}\bigr)^{\frac12}\).
The scaled real polarization identity reads
\[
 \mathfrak B(u,v)
 =\frac14\mathfrak B(au+a^{-1}v,au+a^{-1}v)-\frac14\mathfrak B(au-a^{-1}v,au-a^{-1}v).
\]
Since \(\|au\pm a^{-1}v\|_{Y_T}
\le2\bigl(\|u\|_{Y_T}\|v\|_{Y_T}\bigr)^{\frac12}\), it follows that
\[
        \|\mathfrak B(u,v)\|_{Y_T}
        \le4C\|u\|_{Y_T}\|v\|_{Y_T}.
\]
This recovers the same contradiction on free trajectories and excludes
both unpolarized estimates.
\end{proof}

The obstruction extends to tensor-product spaces. In
Theorem~\ref{thm:5.1.1-quadratic-iteration-obstruction}, replace the target
in \emph{(i)} by \(C([0,T];X^{s_x-\delta_x,r_z})\) and the data norm in
\emph{(ii)} by \(\|H\|_{X^{s_x,s_z}}\), where
\(s_x,s_z,r_z\in\mathbb R\) and \(0\le\delta_x<1\).
Together with \emph{(iii)}, these conditions would imply, for
\(H_1,H_2\in X^\infty\) and \(t_0\in(0,T]\),
\[
 \|\mathcal B_{t_0}(H_1,H_2)\|_{X^{s_x-\delta_x,r_z}}
 \le C_EC_BC_L^2\|H_1\|_{X^{s_x,s_z}}\|H_2\|_{X^{s_x,s_z}},
\]
contrary to Corollary~\ref{cor:4.2.2-anisotropic-picard-obstruction}.
Symmetrization and scaled real polarization again exclude the ordered
and self-interaction estimates.

We next impose existence and Duhamel-consistency assumptions to relate
the Picard obstruction to trajectory-map differentiability. For
\(s>2\) and \(0<\kappa<\gamma\), the set \(X_\kappa^s\) contains a ball
about the origin. By \eqref{eq:4.1.11}, the test data
\(\varepsilon_1H+\varepsilon_2G_M\) belong to this ball and to any
prescribed open neighborhood of zero in \(X^s\), uniformly in \(M\),
when \(|\varepsilon_1|+|\varepsilon_2|\) is sufficiently small.

\begin{corollary}
\label{cor:5.1.1-conditional-non-c2}
Let \(s>2\), \(0\le\delta<1\), \(T>0\), and
\(\kappa\in(0,\gamma)\). Suppose that \(\mathcal O\) is an open
neighborhood of the origin in \(X^s\), with
\(\mathcal O\subset X_\kappa^s\), and that a trajectory map
\[
        \mathcal S:\mathcal O
        \longrightarrow C([0,T];X^{s-\delta})
\]
satisfies the following properties:
\begin{enumerate}[(i)]
\item \(\mathcal S(U_0)(0)=U_0\) for every \(U_0\in\mathcal O\);
\item for every smooth \(U_0\in\mathcal O\), the trajectory
      \(\mathcal S(U_0)\) agrees with a classical solution of
      \eqref{eq:2.1.9};
\item for every such \(U_0\), the trajectory satisfies the Duhamel identity
      associated with \eqref{eq:2.1.11}.
\end{enumerate}
Then \(\mathcal S\) is not twice Fr\'echet differentiable at the origin
as a map between the real Banach spaces \(X^s\) and
\(C([0,T];X^{s-\delta})\). If it were twice differentiable, its second
differential would satisfy
\begin{equation}
\label{eq:5.1.4}
 D^2\mathcal S(0)[H_1,H_2](t)
 =\mathcal B_t(H_1,H_2)
\end{equation}
for every \(t\in(0,T]\) and every \(H_1,H_2\in X^\infty\), contrary to
Theorem~\ref{thm:2.2.1-second-picard-obstruction}.
\end{corollary}

\begin{proof}
Set \(r:=s-\delta>1\). We first establish estimates in \(X^{-2}\)
that justify passage to the limit in the Duhamel identity. For smooth
\(f,g\) with zero vertical mean, H\"older's inequality and
\eqref{eq:2.1.17} imply
\[
 \|b(f,g)\|_{L^1}
 \le\|\partial_z^{-1}\partial_xf\|_{L^2}\|\partial_zg\|_{L^2}
    +\|f\|_{L^2}\|\partial_xg\|_{L^2}
 \le C_r\|f\|_{H^r}\|g\|_{H^r}.
\]
Since \(L^1(\Omega)\hookrightarrow H^{-2}(\Omega)\) by duality and
\(\mathcal P\) is bounded on \(H^{-2}\),
\begin{equation}
\label{eq:5.1.5}
        \|\mathcal N_2(U_1,U_2)\|_{X^{-2}}
        \le C_r
        \|U_1\|_{X^r}\|U_2\|_{X^r}.
\end{equation}

For the cubic term, vertical averaging and the algebra property of
\(H^r(\Omega)\) ensure that
\[
 \|\partial_x\overline{v_1\theta_2}\|_{L^2(\mathbb T_x)}
 \le C_r\|v_1\theta_2\|_{H^r}
 \le C_r\|v_1\|_{H^r}\|\theta_2\|_{H^r}.
\]
View this factor as a function independent of \(z\). Applying
H\"older's inequality to its products with
\(\partial_zv_3,\partial_z\theta_3\), followed by the embedding
\(L^1\hookrightarrow H^{-2}\), we obtain
\begin{equation}
\label{eq:5.1.6}
 \|\mathcal N_3(U_1,U_2,U_3)\|_{X^{-2}}
 \le C_r\|U_1\|_{X^r}\|U_2\|_{X^r}\|U_3\|_{X^r}.
\end{equation}
The projection removes the vertical mean of the quadratic output.
The cubic output has zero vertical mean because its scalar factor is
independent of \(z\) and each remaining factor is a vertical derivative.
Thus density extends both maps continuously from \(X^r\) inputs to
\(X^{-2}\), with the bounds \eqref{eq:5.1.5}--\eqref{eq:5.1.6}.

Suppose, for a contradiction, that \(\mathcal S\) is twice Fr\'echet
differentiable at zero. By \emph{(i)} and \emph{(ii)},
\(\mathcal S(0)\) is a classical solution with zero initial data.
The energy identity \eqref{eq:2.1.10} therefore forces
\(\mathcal S(0)=0\).

For smooth \(U_0\in\mathcal O\), write \(U(t):=\mathcal S(U_0)(t)\).
Property~\emph{(iii)} reads
\begin{equation}
\label{eq:5.1.7}
 U(t)=e^{t\mathcal L}U_0+
 \int_0^t e^{(t-\tau)\mathcal L}
 \bigl[\mathcal N_2(U(\tau),U(\tau))
      +\mathcal N_3(U(\tau),U(\tau),U(\tau))\bigr] \,\mathrm d\tau.
\end{equation}
We interpret this identity in \(X^{-2}\). By
Proposition~\ref{prop:2.1.1-unitary-linear-evolution}, the Duhamel
operator is bounded on \(C([0,T];X^{-2})\), with operator norm at
most \(T\). Hence limits in this space pass through the integral.

For a fixed \(H\in X^\infty\), we have
\(\varepsilon H\in\mathcal O\) whenever the real parameter
\(\varepsilon\) is sufficiently small. Differentiability at zero and
\(\mathcal S(0)=0\) imply
\[
        \|\mathcal S(\varepsilon H)\|_{C([0,T];X^r)}
        \le C_H|\varepsilon|.
\]
Using this bound in \eqref{eq:5.1.5}--\eqref{eq:5.1.7}, we obtain
\[
 \|\mathcal S(\varepsilon H)-\varepsilon e^{\cdot\mathcal L}H
     \|_{C([0,T];X^{-2})}
 \le C_{H,T}\bigl(\varepsilon^2+|\varepsilon|^3\bigr).
\]
After division of \eqref{eq:5.1.7} by \(\varepsilon\), the preceding
estimate identifies the limit in \(C([0,T];X^{-2})\) as
\[
        D\mathcal S(0)[H](t)=e^{t\mathcal L}H.
\]
Both sides lie in \(C([0,T];X^r)\). The injectivity of
\(X^r\hookrightarrow X^{-2}\) therefore establishes the identity in
\(C([0,T];X^r)\).

To identify the second differential, fix \(H_1,H_2\in X^\infty\),
write \(V_j=e^{\cdot\mathcal L}H_j\), and, for sufficiently small real
\(\varepsilon\), set
\[
 U_\varepsilon^{12}:=\mathcal S(\varepsilon(H_1+H_2))
 \quad\text{and}\quad
 U_\varepsilon^j:=\mathcal S(\varepsilon H_j),\qquad j=1,2.
\]
For each fixed direction \(H\in\{H_1,H_2,H_1+H_2\}\), the
second-order Taylor expansion in \(C([0,T];X^r)\) is
\[
 \mathcal S(\varepsilon H)
 =\varepsilon e^{\cdot\mathcal L}H
  +\frac{\varepsilon^2}{2}D^2\mathcal S(0)[H,H]
  +o(\varepsilon^2).
\]
Subtracting these expansions and using symmetry of the second
differential, we find
\begin{equation}
\label{eq:5.1.8}
 \frac{U_\varepsilon^{12}-U_\varepsilon^1-U_\varepsilon^2}
      {\varepsilon^2}
 \longrightarrow
 D^2\mathcal S(0)[H_1,H_2]
 \quad\text{in }C([0,T];X^r).
\end{equation}
First-order differentiability also gives, in \(C([0,T];X^r)\),
\[
 \frac{U_\varepsilon^{12}}{\varepsilon}\longrightarrow V_1+V_2
 \quad\text{and}\quad
 \frac{U_\varepsilon^j}{\varepsilon}\longrightarrow V_j,
 \qquad j=1,2.
\]
By the bilinear bound \eqref{eq:5.1.5}, these limits imply, in
\(C([0,T];X^{-2})\),
\[
 \frac{\mathcal N_2(U_\varepsilon^{12},U_\varepsilon^{12})
       -\sum_{j=1}^2\mathcal N_2(U_\varepsilon^j,U_\varepsilon^j)}
      {\varepsilon^2}
 \longrightarrow\mathcal N_2(V_1,V_2)+\mathcal N_2(V_2,V_1).
\]
All three trajectories have size \(O(|\varepsilon|)\) in
\(C([0,T];X^r)\). By \eqref{eq:5.1.6}, each cubic term is
\(O(|\varepsilon|)\) in \(C([0,T];X^{-2})\) after division by
\(\varepsilon^2\).

Subtracting the three Duhamel identities cancels the linear terms.
Divide by \(\varepsilon^2\), pass to the limit, and use
\eqref{eq:5.1.8} to conclude, in \(C([0,T];X^{-2})\), that
\[
\begin{aligned}
 D^2\mathcal S(0)[H_1,H_2](t)
 ={}&\int_0^te^{(t-\tau)\mathcal L}
 \bigl[\mathcal N_2(V_1(\tau),V_2(\tau))
      +\mathcal N_2(V_2(\tau),V_1(\tau))\bigr] \,\mathrm d\tau\\
 ={}&\mathcal B_t(H_1,H_2).
\end{aligned}
\]
For smooth \(H_1,H_2\), both sides lie in \(C([0,T];X^r)\).
Thus \eqref{eq:5.1.4} holds in the asserted space, by injectivity of
the embedding into \(C([0,T];X^{-2})\).

Evaluation at any fixed \(t_0\in(0,T]\) is bounded. Thus
\eqref{eq:5.1.4} and the boundedness of \(D^2\mathcal S(0)\) extend
\(\mathcal B_{t_0}\) from the dense subspace
\(X^\infty\times X^\infty\) to a bounded real-bilinear map from
\(X^s\times X^s\) to \(X^{s-\delta}\), contrary to
Theorem~\ref{thm:2.2.1-second-picard-obstruction}.
\end{proof}

\subsection{Quadratic changes of variables}

The resonant term identified in
Theorem~\ref{thm:2.2.4-mean-resonant-averaging} raises the question of
whether a quadratic change of variables can remove the fixed-time
obstruction.

For \(s\in\mathbb R\), consider a time-independent \(C^2\) local
change of variables \(U=\Psi(V)\) on the real space \(X^s\), with
\(\Psi(0)=0\) and \(D\Psi(0)=I\). Its second differential
\(\mathcal A:=D^2\Psi(0)\) is a bounded symmetric bilinear map from
\(X^s\times X^s\) to \(X^s\), and the Taylor expansion is
\[
 \Psi(V)=V+\frac12\mathcal A(V,V)+o(\|V\|_{X^s}^2)
 \qquad\text{in }X^s.
\]
Bounded quadratic corrections alter the second Picard operator only
by endpoint terms. The next result also identifies the resonant
coefficients preserved by Fourier-multiplier corrections.

\begin{corollary}
\label{cor:5.2.1-quadratic-change-obstruction}
Let \(s\in\mathbb R\), \(0\le\delta<1\), and let
\(\mathcal A:X^s\times X^s\to X^s\) be a bounded symmetric real-bilinear
map. For \(H_1,H_2\in X^\infty\), define in \(X^{s-1}\)
\begin{equation}
\label{eq:5.2.1}
\begin{aligned}
 \widetilde{\mathcal N}_2^{\mathcal A}(H_1,H_2)
 :={}&\mathcal N_2(H_1,H_2)+\mathcal N_2(H_2,H_1)
       +\mathcal L\mathcal A(H_1,H_2)\\
 &-\mathcal A(\mathcal LH_1,H_2)
       -\mathcal A(H_1,\mathcal LH_2).
\end{aligned}
\end{equation}
Define the associated quadratic Duhamel operator by the
\(X^{s-1}\)-valued integral
\begin{equation}
 \mathcal B_t^{\mathcal A}(H_1,H_2)
 :=\int_0^te^{(t-\tau)\mathcal L}
       \widetilde{\mathcal N}_2^{\mathcal A}
       \bigl(e^{\tau\mathcal L}H_1,e^{\tau\mathcal L}H_2\bigr)
       \,\mathrm d\tau.
\end{equation}
These operators satisfy
\begin{equation}
\label{eq:5.2.3}
 \mathcal B_t^{\mathcal A}(H_1,H_2)
 =\mathcal B_t(H_1,H_2)+e^{t\mathcal L}\mathcal A(H_1,H_2)
       -\mathcal A\bigl(e^{t\mathcal L}H_1,e^{t\mathcal L}H_2\bigr).
\end{equation}
For smooth inputs, the output lies in \(X^{s-\delta}\), with
\begin{equation}
\label{eq:5.2.4}
 \|\mathcal B_t^{\mathcal A}(H_1,H_2)-\mathcal B_t(H_1,H_2)\|_{X^{s-\delta}}
 \le 2\|\mathcal A\|\,\|H_1\|_{X^s}\|H_2\|_{X^s}
 \qquad(t\in\mathbb R),
\end{equation}
where \(\|\mathcal A\|\) denotes its bilinear operator norm. Consequently,
for every fixed \(t\ne0\), \(\mathcal B_t^{\mathcal A}\) has no bounded
real-bilinear extension from \(X^s\times X^s\) to \(X^{s-\delta}\).

For the coefficient statement, assume instead that \(\mathcal A\) is
a time-independent symmetric real-bilinear Fourier multiplier mapping
pairs of constrained trigonometric polynomials to constrained
trigonometric polynomials, without any boundedness assumption.
Define its finite wave coefficients
\(a_{\sigma\mu\nu}(k;k_1,k_2)\) by \eqref{eq:3.7}, with \(\mathcal A\)
in place of \(\mathcal N_2\). Under the formal quadratic change
\(U=V+\frac12\mathcal A(V,V)\), the transformed polarized coefficient is
\begin{equation}
\label{eq:5.2.5}
 \widetilde q^{\mathcal A}_{\sigma\mu\nu}
 =\widetilde q_{\sigma\mu\nu}
    +i\Phi_{\sigma\mu\nu}a_{\sigma\mu\nu},
 \qquad k=k_1+k_2\in K.
\end{equation}
Every active exact resonant coefficient is therefore unchanged.
In particular, no Fourier-multiplier change of this form eliminates
the full quadratic term.
\end{corollary}

\begin{proof}
Equations \eqref{eq:2.1.6} and \eqref{eq:2.1.17} imply that
\(\mathcal L:X^s\to X^{s-1}\) is bounded. Thus, for bounded
\(\mathcal A\), expression \eqref{eq:5.2.1} is well defined on smooth
inputs even if \(\mathcal A(H_1,H_2)\) is not smooth. Set
\(V_j(\tau):=e^{\tau\mathcal L}H_j\). These trajectories lie in
\(C^1(\mathbb R;X^s)\) and satisfy \(\partial_\tau V_j=\mathcal LV_j\).
Bounded bilinearity also makes
\(\tau\mapsto\mathcal A(V_1(\tau),V_2(\tau))\) a \(C^1\) map into
\(X^s\). Since \(X^s\) lies in the generator domain on \(X^{s-1}\),
the product rule and
Proposition~\ref{prop:2.1.1-unitary-linear-evolution} imply, in
\(X^{s-1}\),
\[
 \frac{\mathrm d}{\mathrm d\tau}
 \bigl[e^{(t-\tau)\mathcal L}\mathcal A(V_1(\tau),V_2(\tau))\bigr]
 = e^{(t-\tau)\mathcal L}
 \bigl[-\mathcal L\mathcal A(V_1,V_2)
 +\mathcal A(\mathcal LV_1,V_2)
             +\mathcal A(V_1,\mathcal LV_2)\bigr].
\]
Integrating from \(0\) to \(t\) and comparing with \eqref{eq:2.2.1}
proves \eqref{eq:5.2.3} for every \(t\in\mathbb R\), with the usual
orientation of the integral when \(t<0\).

By \eqref{eq:2.1.15}, the embedding \(X^s\hookrightarrow X^{s-\delta}\)
has norm at most one. The triangle inequality and the isometry of
\(e^{t\mathcal L}\) control the two endpoint terms in
\eqref{eq:5.2.3} and establish \eqref{eq:5.2.4}. The same identity
proves the asserted regularity, since \(\mathcal B_t(H_1,H_2)\) is
smooth for smooth inputs.

For the normalized real inputs \(H,G_M\) in
\eqref{eq:4.1.9}--\eqref{eq:4.1.11}, the reverse triangle inequality,
\eqref{eq:2.2.2}, and \eqref{eq:5.2.4} imply
\[
 \|\mathcal B_t^{\mathcal A}(H,G_M)\|_{X^{s-\delta}}
 \ge c_{s,\delta}|t|M^{2(1-\delta)}-2\|\mathcal A\|.
\]
Both input norms equal one, while the lower bound diverges as
\(M\to\infty\) for every fixed \(t\ne0\). Hence no bounded extension
exists. The inverse change reverses the sign of the endpoint correction
without affecting the obstruction.

For the coefficient identity, we work on trigonometric polynomials.
By symmetry of \(\mathcal A\), the formal substitution
\(U=V+\frac12\mathcal A(V,V)\) satisfies
\[
\partial_tU=\partial_tV+\mathcal A(V,\partial_tV).
\]
Collecting the quadratic terms in \eqref{eq:2.1.11} and polarizing
recovers \eqref{eq:5.2.1}. By \eqref{eq:2.1.28} and
\eqref{eq:2.1.36}, the three terms involving \(\mathcal L\) have
polarized wave coefficient \(i\Phi_{\sigma\mu\nu}a_{\sigma\mu\nu}\).
Together with Lemma~\ref{lem:3.1-quadratic-branch-multipliers}, this
proves \eqref{eq:5.2.5}.

Elimination of the full quadratic term would require the homological
equation
\begin{equation}
\label{eq:5.2.6}
 i\Phi_{\sigma\mu\nu}(k;k_1,k_2)
 a_{\sigma\mu\nu}(k;k_1,k_2)
 =-\widetilde q_{\sigma\mu\nu}(k;k_1,k_2)
\end{equation}
for every admissible interaction. At the active exact resonances in
Corollaries~\ref{cor:4.1.1-active-high-low-resonances}
and~\ref{cor:4.2.1-diagnostic-mixed-resonances}, its left-hand side
vanishes for every finite \(a_{\sigma\mu\nu}\), but its right-hand
side is nonzero. This excludes elimination of the full quadratic term
by such a Fourier-multiplier change.
\end{proof}

\subsection{Dynamics on Fourier rays}

On the transparent Fourier rays classified in
Theorem~\ref{thm:2.2.3-maximal-transparent-rays}, the nonlinear terms
vanish. We extend the ordered cancellations to \(H^s\) data and
solve the resulting constant-coefficient linear system.

\begin{corollary}
\label{cor:5.3.1-fourier-ray-dynamics}
Let \(s>2\), let \(k_0=(n_0,m_0)\) be primitive with \(m_0\ne0\), and set
\[
        \lambda_0:=\frac{L_zn_0}{L_xm_0}.
\]
For every real-valued \(U_0=(v_0,\theta_0)\in\mathcal R_{k_0}^s\), the
closed system \eqref{eq:2.1.9} has the global ray-supported solution
\begin{equation}
\label{eq:5.3.1}
 \begin{pmatrix}
  \sqrt\gamma\,v(t)\\[1mm]
  \sqrt{g\beta}\,\theta(t)
 \end{pmatrix}
 =
 \begin{pmatrix}
  \cos(N\lambda_0t)&-\sin(N\lambda_0t)\\[1mm]
  \sin(N\lambda_0t)& \cos(N\lambda_0t)
 \end{pmatrix}
 \begin{pmatrix}
  \sqrt\gamma\,v_0\\[1mm]
  \sqrt{g\beta}\,\theta_0
 \end{pmatrix}.
\end{equation}
This is the unique solution with initial data \(U_0\) in
\(C^1(\mathbb R;\mathcal R_{k_0}^s)\), and
\(\|U(t)\|_{X^s}=\|U_0\|_{X^s}\). It solves the original system
\eqref{eq:1.2.1} with diagnostic variables
\begin{equation}
\label{eq:5.3.2}
 w(t)=-\lambda_0v(t),
\end{equation}
and
\begin{equation}
\label{eq:5.3.3}
 p(t)=g\beta\partial_z^{-1}\theta(t).
\end{equation}
The solution is periodic in time with period
\(\frac{2\pi}{N|\lambda_0|}\) when \(n_0\ne0\), and stationary when
\(n_0=0\). Smoothness of the initial data persists globally.
\end{corollary}

\begin{proof}
Let \(U_\ell\in\mathcal R_{k_0}^s\), \(\ell=1,2,3\).
The Fourier cutoffs \(\mathcal P_{\le j}\) from \eqref{eq:4.2.10}
preserve the ray and satisfy
\(\mathcal P_{\le j}U_\ell\in\mathcal R_{k_0}^\infty\) and
\(\mathcal P_{\le j}U_\ell\to U_\ell\) in \(X^s\).
By Theorem~\ref{thm:2.2.3-maximal-transparent-rays},
\[
 \mathcal N_2(\mathcal P_{\le j}U_1,\mathcal P_{\le j}U_2)=0
 \quad\text{and}\quad
 \mathcal N_3(\mathcal P_{\le j}U_1,\mathcal P_{\le j}U_2,
             \mathcal P_{\le j}U_3)=0.
\]
The multilinear bounds in
Proposition~\ref{prop:3.1-quadratic-endpoint-bounds} allow passage to
the limit in \(X^{s-1}\). Fourier convergence extends
the scalar multiplier identity \eqref{eq:4.1.21}. Continuity of
averaging and multiplication in \(H^s\) likewise extends
\eqref{eq:4.1.23} to the original inputs in \(H^{s-1}\).

The linear operator preserves \(\mathcal R_{k_0}^s\). On this space,
\eqref{eq:4.1.21} reduces the closed system to
\begin{equation}
\label{eq:5.3.4}
 \partial_tv=-g\beta\lambda_0\theta,
\end{equation}
and
\begin{equation}
\label{eq:5.3.5}
 \partial_t\theta=\gamma\lambda_0v.
\end{equation}
In the energy variables, this system has the rotation formula
\eqref{eq:5.3.1}. The rotation preserves the Fourier support and the
\(X^s\) norm, so the nonlinear terms vanish along the entire trajectory.
Since the restricted generator is bounded on \(\mathcal R_{k_0}^s\),
we have \(U\in C^1(\mathbb R;\mathcal R_{k_0}^s)\).
Every solution in this class satisfies
\eqref{eq:5.3.4}--\eqref{eq:5.3.5}, which also proves uniqueness.

Equation \eqref{eq:4.1.23} shows that \(\overline{v\theta}\) and
\(\overline{v^2}\) are independent of \(x\). The reconstruction
formulas \eqref{eq:2.1.2}--\eqref{eq:2.1.3}, together with
\eqref{eq:4.1.21}, then reduce to
\eqref{eq:5.3.2}--\eqref{eq:5.3.3}. The zero vertical mean of
\(\partial_z^{-1}\theta\) ensures the prescribed pressure normalization.

For \(f=v,\theta\), the ray identity also implies
\[
 v\partial_xf+w\partial_zf
 =v(\partial_xf-\lambda_0\partial_zf)=0.
\]
Moreover, \(\partial_xp=g\beta\lambda_0\theta\),
\(\partial_zp=g\beta\theta\), and
\(\partial_xv+\partial_zw=0\). Together with
\eqref{eq:5.3.4}--\eqref{eq:5.3.5}, these identities verify
\eqref{eq:1.2.1} in \(H^{s-1}\).
Formula \eqref{eq:5.3.1} directly shows periodicity when \(n_0\ne0\),
stationarity when \(n_0=0\), and persistence of smoothness.
\end{proof}

To characterize uniform strict stratification for the solutions in
Corollary~\ref{cor:5.3.1-fourier-ray-dynamics}, set
\begin{equation}
\label{eq:5.3.6}
 M_0(x,z):=\left(\frac{\gamma}{g\beta}|\partial_zv_0(x,z)|^2
                    +|\partial_z\theta_0(x,z)|^2\right)^{\frac12}.
\end{equation}
For \(n_0\ne0\), this amplitude determines the space-time essential
infimum of the total vertical temperature gradient.

\begin{proposition}
\label{prop:5.3.1-uniform-stratification}
Let \(U(t)\) be the solution in
Corollary~\ref{cor:5.3.1-fourier-ray-dynamics}, and let
\(\kappa\in(0,\gamma)\). If \(n_0\ne0\), then
\begin{equation}
\label{eq:5.3.7}
 \operatorname*{ess\,inf}_{(t,x,z)\in\mathbb R\times\Omega}
 \bigl(\gamma+\partial_z\theta(t,x,z)\bigr)
 =\gamma-\|M_0\|_{L^\infty}.
\end{equation}
Consequently,
\begin{equation}
\label{eq:5.3.8}
 \gamma+\partial_z\theta(t,x,z)\ge\kappa
 \quad\text{for a.e. }(t,x,z)\in\mathbb R\times\Omega
\end{equation}
if and only if
\(\|M_0\|_{L^\infty}\le\gamma-\kappa\).
If \(n_0=0\), then \(U(t)=U_0\), and \eqref{eq:5.3.8} holds if and only if
\[
        \operatorname*{ess\,inf}_{\Omega}
        (\gamma+\partial_z\theta_0)\ge\kappa.
\]
\end{proposition}

\begin{proof}
The temperature component of \eqref{eq:5.3.1} satisfies
\begin{equation}
 \partial_z\theta(t)
 =\sqrt{\frac{\gamma}{g\beta}}\sin(N\lambda_0t)\,\partial_zv_0
  +\cos(N\lambda_0t)\,\partial_z\theta_0.
\end{equation}
Assume \(n_0\ne0\). By Cauchy--Schwarz and \eqref{eq:5.3.6},
\(\partial_z\theta(t,x,z)\ge-M_0(x,z)\). The assumption \(s>2\)
ensures continuity of the initial vertical derivatives. Thus \(M_0\)
attains its maximum at some \((x_*,z_*)\).
If \(M_0\equiv0\), identity \eqref{eq:5.3.7} is immediate.
Otherwise choose \(\tau_*\) such that
\[
 \sqrt{\frac{\gamma}{g\beta}}\partial_zv_0(x_*,z_*)\sin\tau_*
 +\partial_z\theta_0(x_*,z_*)\cos\tau_*
 =-M_0(x_*,z_*).
\]
Since \(N\lambda_0\ne0\), choose \(t_*\) so that
\(N\lambda_0t_*\equiv\tau_*\pmod{2\pi}\). At this time,
\[
 \partial_z\theta(t_*,x_*,z_*)=-\|M_0\|_{L^\infty}.
\]
By joint continuity, values arbitrarily close to this minimum occur
on sets of positive space-time measure. Together with
\(\partial_z\theta\ge-\|M_0\|_{L^\infty}\), this establishes
\eqref{eq:5.3.7} and the criterion \eqref{eq:5.3.8}. For \(n_0=0\),
stationarity in Corollary~\ref{cor:5.3.1-fourier-ray-dynamics}
proves the remaining assertion.
\end{proof}


\end{document}